\documentclass[a4paper,11pt]{amsart}
\usepackage{mathrsfs}
\usepackage[all]{xy}
\usepackage{amsmath,amssymb,amscd,bbm,amsthm,mathrsfs,dsfont}
\usepackage{graphicx}
\newtheorem{thm}{Theorem}[section]
\newtheorem{lem}{Lemma}[section]
\newtheorem{cor}{Corollary}[section]
\newtheorem{prop}{Proposition}[section]
\newtheorem{rem}{Remark}[section]

\begin{document}
\numberwithin{equation}{section}

\title[Conformal Qc geometry, Spherical qc Manifolds and Convex Cocompact Subgroups]{On Conformal Qc Geometry,  Spherical Qc Manifolds and Convex
Cocompact Subgroups of ${\rm Sp}{(n+1,1)}$}
\author{Yun Shi and Wei Wang}

\begin{abstract}Conformal qc geometry of spherical qc manifolds 
are investigated. We construct the qc Yamabe operators on  qc manifolds, which are covariant
under the conformal qc transformations. A qc manifold is  scalar positive,  negative or  vanishing if and only if  its qc Yamabe invariant is positive,
negative or zero, respectively. On a scalar positive spherical qc manifold, we can construct the Green function of the qc Yamabe operator, which can be
applied to   construct a conformally  invariant tensor.  It becomes a spherical qc metric if the qc positive mass conjecture is true. Conformal qc  geometry
of spherical qc manifolds can be applied to study convex cocompact subgroups of ${\rm Sp}(n+1,1).$ On a spherical qc manifold   constructed from such
a discrete subgroup, we construct a    spherical qc metric  of Nayatani type. As a corollary, we prove that such a spherical qc manifold is scalar positive,
negative or  vanishing if and only if the Poincar\'e  critical exponent of the discrete subgroup is less than,  greater than or equal to  $2n+2$,
respectively.
\end{abstract}
\thanks{Supported by National Nature Science Foundation in China (No. 11171298)\\
Department of Mathematics, Zhejiang University,  Zhejiang 310027, P.R. China
E-mail: hzxjhs1987@163.com;
Department of Mathematics, Zhejiang University,  Zhejiang 310027, P.R. China
E-mail: wwang@zju.edu.cn}
\maketitle

\tableofcontents

\section{Introduction}
A locally conformally flat manifold is a Riemannian manifold which is locally conformally equivalent to the Euclidean space. Its complex counterpart is
the spherical CR manifold, which is a pseudohermitian manifold locally conformally equivalent to  the  Heisenberg group. A conformal  class of the locally
conformally flat manifolds (or spherical CR manifolds) can be described as a manifold whose  coordinate charts are given by open subsets of the Euclidean
space (or the Heisenberg group) and elements of ${\rm SO}(n+1,1)$ (or ${\rm SU}(n+1,1)$) as transition maps.  In this paper, we will investigate their
quaternionic counterpart: the spherical qc manifolds.

The notion of a qc manifold was introduced  by Biquard \cite{biquard} as the quaternionic counterpart of the notion of a  pseudohermitian  manifold in the
CR geometry. A qc manifold is denoted by  $(M,g,\mathbb{Q}),$ where $M$ is a $(4n+3)$-dimensional manifold, $g$ is the Carnot-Carath\'eodory metric on a
codimension three distribution $H$ and $\mathbb{Q}$ is a rank-three bundle of quaternionic structures. As the Tanaka-Webster connection in the CR
geometry, the Biquard connection is the canonical connection for qc manifolds \cite{biquard}. $(M,g,\mathbb{Q})$ is said to be conformal to
$(M,\tilde{g},\mathbb{Q}),$ if $\tilde{g}=\phi g$ for some positive function $\phi.$ A qc manifold $(M,g,\mathbb{Q})$ is called \emph{spherical} if it is
locally conformally  qc equivalent to an open set of  the quaternionic Heisenberg group. We are interested in the conformal classes of spherical qc
manifolds. A conformal class can be described as a manifold   whose  coordinate charts are given by open subsets of the quaternionic Heisenberg group and
elements of ${\rm Sp}(n+1,1)$ as transition maps. This is a topological description of a conformal class of spherical qc manifolds. The purpose of this
paper is to investigate conformal qc geometry for spherical qc manifolds.

Spherical qc manifolds are abundant. We can construct spherical qc manifolds by taking  connected sums. The connected sum of two spherical qc manifolds is
constructed as follows (cf. $\S$ 2.5). Suppose that $M_{(1)}$ and $M_{(2)}$ are two   spherical qc manifolds with one puncture $\eta_{i}\in
M_{(i)},i=1,2$, each. Let  $U_{1}$ and $U_{2}$ be   neighborhoods of $\eta_{1}$ and $\eta_{2}$, respectively. Now identify  $U_{i}$  with the ball of
quaternionic   Heisenberg group with radius $2$, centered at origin. We remove the closed balls with radius $t<1$. Let $U_{i}(t,1),\ i=1,2,$ be
corresponding rings with inner radius $t$ and outer radius $1$.  There exist conformal qc transformations mapping $U_{1}(t,1)$ to $U_{2}(t,1),$ which
identify the inner boundary of $U_{1}(t,1)$ with the outer boundary $U_{2}(t,1)$ and vise versa. We glue $M_{1}$ and $M_{2}$ by such a  transformation to
get a new spherical qc manifold.

As in the locally conformally flat case and the spherical CR case, {convex cocompact} subgroups of ${\rm Sp}(n+1,1)$ provide lots of examples of spherical qc
manifolds. ${\rm Sp}(n+1,1)$ acts isometrically on the unit ball $B^{4n+4}$ of the $(n+1)$-dimensional quaternionic space, when the ball is equipped with
the quaternionic  hyperbolic metric. Also ${\rm Sp}(n+1,1)$ acts conformally on its boundary, the sphere $S^{4n+3}$ with the standard qc metric. Let
$\Gamma$ be a convex cocompact subgroup of ${\rm Sp}(n+1,1).$ The \emph{limit set} of $\Gamma$ is
$$
\Lambda(\Gamma)=\overline{\Gamma{q}}\cap S^{4n+3},
$$ for $q\in B^{4n+4},$ where $\overline{\Gamma{q}}$ is the closure of the orbit of $q$ under $\Gamma,$ and
$$\Omega(\Gamma)= S^{4n+3}\backslash\Lambda(\Gamma)$$
is the maximal open set where $\Gamma$ acts discontinuously.  It is known that $\Omega(\Gamma)/\Gamma$ is a compact spherical qc manifold  when $\Gamma$
is a convex cocompact subgroup of ${\rm Sp}(n+1,1)$. Modelled on the theory of Riemann surfaces, the basic problems for locally conformally flat manifolds
are their classification and the moduli space. It is so for spherical CR manifolds and spherical qc manifolds. It is well known that each   compact
Riemann surface has a metric with constant curvature, which only depends on the conformal  class of Riemann surfaces. This invariant metric was
generalized by Habermann-Jost \cite{habermann} \cite{habermann1} to the locally conformally flat case. They constructed a canonical metric on each scalar
positive locally conformally flat manifold, which only depends on the conformal class of such manifolds. This construction was generalized to the
spherical CR case by the second author \cite{wang1}.

The uniformization problem is to show each scalar positive spherical qc manifold can be constructed from some convex cocompact subgroup of ${\rm
Sp}(n+1,1)$. The problem was solved in \cite{yau} for locally conformally flat manifolds and in \cite{paul} for spherical CR manifolds. See \cite{Izeki}
and \cite{Wang4} for discussion of their moduli spaces.

In Section 2, we recall the definitions  of a  qc manifold, the  Biquard connection and the ball model of the quaternionic hyperbolic space. For our
purpose later, we describe explicitly the flat qc structure on the quaternionic Heisenberg group and the standard spherical qc metric on the sphere. The
group ${\rm Sp}(n+1,1)$ is the group of all $(n+2)\times(n+2)$ quaternionic matrices which preserve the following hyperhermitian form:
\begin{align}\label{2.31}
Q(q,p)=-{q}_{1}\bar{p}_{1}-\cdots-{q}_{n+1}\bar{p}_{n+1}+
{q}_{n+2}\bar{p}_{n+2},
\end{align}
where $q=(q_{1},\cdots,q_{n+2}),\ p=(p_{1},\cdots,p_{n+2})\in \mathbb{H}^{n+2}$. In this paper we only consider left quaternionic vector spaces. So a
matrix acts on a vector from right.  ${\rm Sp}(n+1,1)$ acts  isometrically on the ball model of the quaternionic hyperbolic space and  conformally on the
sphere with the standard qc structure. We write down such a conformal action explicitly by calculating the conformal factor in terms of matrix elements of
an element of ${\rm Sp}(n+1,1)$. We also show that a qc manifold $(M,g,\mathbb{Q})$ is  {spherical} if and only if  it is a manifold with coordinates
charts $\{(U_{i},\phi_{i})\}$, where $\phi_{i}:U_{i}\rightarrow S^{4n+3}$  and transition  maps are given by the induced action of elements of ${\rm
Sp}(n+1,1).$

Since topologically defined spherical qc manifolds depends on the conformal class of qc metrics, we need to study conformal qc geometry. In particular, we
are interested in constructing conformal invariants. We will construct conformally covariant operators, since some quantity associated to such a operator
will provide us a conformal invariant. In this way we can obtain invariants for topologically defined spherical qc manifolds. Given a qc manifold
$(M,g,\mathbb{Q}),$ let $\Delta_{g,\mathbb{Q}}$ be the SubLaplacian operator and let  $s_{g,\mathbb{Q}}$ be its scalar curvature.
In Section 3, we construct the \emph{qc Yamabe operator}
$$L_{g,\mathbb{Q}}=b_{n}\Delta_{g,\mathbb{Q}}+s_{g,\mathbb{Q}},$$  where $b_{n}=4\frac{Q+2}{Q-2},$ and $Q=4n+6$ is the homogeneous dimension of $M.$  This
is a conformally covariant operator. Namely, under the conformal change
\begin{align}\label{1.0}
\tilde{g}=\phi^{\frac{4}{Q-2}}g,
\end{align}  it satisfies the following  transformation law
\begin{align*}
L_{\tilde{g},\mathbb{Q}}f=\phi^{-\frac{Q+2}{Q-2}}L_{g,\mathbb{Q}}(\phi f),\nonumber
\end{align*}
for any smooth real function $f.$ Denote by $G_{g,\mathbb{Q}}(\xi,\cdot)$ the Green function of the  qc Yamabe operator with the pole at $\xi$, i.e.
$$L_{g,\mathbb{Q}}G_{g,\mathbb{Q}}(\xi,\cdot)=\delta_{\xi},$$
where $\delta_{\xi}$ is the Dirac function at the point $\xi.$

As in the  Riemannian case \cite{yau} and the CR case \cite{wang1}, we will see that for a connected compact qc manifold $(M,g,\mathbb{Q})$,  one and only
one of the following cases holds: there is  a qc metric $\tilde{g}$ conformal to $g$  which have either positive,   negative or zero scalar curvature
everywhere. The manifold is called scalar positive, scalar negative or scalar vanishing respectively.  This is equivalent to its first eigenvalue of the
qc Yamabe operator or the qc Yamabe invariant  is positive, negative or zero, respectively. This is a property of conformal classes. On a scalar positive
spherical qc manifold, the Green function of the qc Yamabe operator $L_{g,\mathbb{Q}}$ always exists, and
\begin{align}\label{222} \rho_{g,\mathbb{Q}}(\xi,\eta)&=\frac{1}{\phi(\xi)\phi(\eta)}
\cdot\frac{C_{Q}}{\|\xi^{-1}\eta\|^{Q-2}},\quad \xi,\eta\in \mathscr{H}^{n},
\end{align}
is its singular part, if we identify a neighborhood of $\xi$ with an open set of the quaternionic Heisenberg group with the qc metric
$g=\phi^{\frac{4}{Q-2}}g_{0}.$ Here $g_{0}$ is the standard qc metric on the quaternionic Heisenberg group $\mathscr{H}^{n}$, $\|\cdot\|$ is the norm on
the quaternionic Heisenberg group and $C_{Q}$ is a positive constant  (\ref{146}). Moreover the limit  \begin{equation}\label{1.3}
\begin{aligned}
\mathcal{A}_{g,\mathbb{Q}}(\xi)&=\lim_{\eta\rightarrow \xi}|G_{g,\mathbb{Q}}(\xi,\eta)-\rho_{g,\mathbb{Q}} (\xi,\eta)|^{\frac{1}{Q-2}}
\end{aligned}\end{equation}
of the nonsingular part of the Green function exists, and $\mathcal{A}_{g,\mathbb{Q}}^{2}g$ is  conformal  invariant, i.e.
$$\mathcal{A}_{\tilde{g},\mathbb{Q}}^{2}\tilde{g}
=\mathcal{A}_{g,\mathbb{Q}}^{2}g,$$
under the conformal transformation  (\ref{1.0}). We also prove the transformation law of Green functions under the  conformal change (\ref{1.0}):
\begin{align*}
G_{\tilde{g},\mathbb{Q}}(\xi,\eta)=\frac{1}{\phi(\xi)\phi(\eta)}
G_{g,\mathbb{Q}}
(\xi,\eta).
\end{align*}

As in the conformal geometry and the CR geometry, we propose the following positive mass conjecture.

The {{qc positive mass conjecture}}: {\it
Let $(M,g,\mathbb{Q})$ be a compact scalar positive  spherical qc manifold with ${\rm dim} M=4n+3$. Then,\\
\noindent 1. For each $\xi\in M$, there exists a local qc diffeomorphism $C_{\xi}$ from a neighborhood of $\xi$ to the quaternionic Heisenberg group
$\mathscr{H}^{n}$ such that $C_{\xi}(\xi)=\infty$ and  $$\left(C_{\xi}^{-1}\right)^{*}\left(G_{g,\mathbb{Q}}(\xi,\cdot)
^{\frac{4}{Q-2}}g\right)
=h^{\frac{4}{Q-2}}g_{0},$$ where  $$h(\eta)=1+A_{g,\mathbb{Q}}(\xi)\|\eta\|^{-Q+2}+O(\|\eta\|^{-Q+1}),$$
near $\infty,$ and $g_{0}$ is the standard qc metric on $\mathscr{H}^{n}.$
$A_{g,\mathbb{Q}}(\xi)$ is called the {qc mass} at the point $\xi$.

\noindent 2. $A_{g,\mathbb{Q}}(\xi)$ is nonnegative and is zero if and only if $(M,g,\mathbb{Q})$ is qc equivalent to the standard sphere.}

Li  gave the statement of the CR positive mass theorem  in \cite{Li}. A complete proof of CR positive mass theorem was given by  Cheng, Chiu  and Yang  in
\cite{paul} recently.

This conjecture implies that $\mathcal{A}_{g,\mathbb{Q}}$ is non-vanishing, so $\mathcal{A}_{g,\mathbb{Q}}^{2}g$ is a conformally invariant qc metric. By
estimating the qc Yamabe invariant of the connected sum, we prove that some connected sum of two scalar positive
spherical qc manifolds is also scalar positive. So scalar positive spherical qc manifolds are abundant.

In Section 5, we recall the definitions of  the convex cocompact subgroups of ${\rm Sp}(n+1,1)$ and the Patterson-Sullivan measure.
The \emph{Poincar\'e critical exponent} $\delta(\Gamma)$ of a discrete subgroup $\Gamma$ is defined as
\begin{align*}
\delta(\Gamma)=\inf\left\{s>0;\sum_{\gamma\in\Gamma}e^{-\frac{1}{2}s\cdot d(p,\gamma (q))}<\infty\right\},
\end{align*}
where $p$ and $q$ are two points in the ball ${B}^{4n+4}$ and $d(\cdot,\cdot)$ is the quaternionic hyperbolic distance on ${B}^{4n+4}$. $\delta(\Gamma)$
is independent of the particular choice of points $p$ and $q$.
Fix a point $q\in{B}^{4n+4}$, the series
$$\sum_{\gamma\in\Gamma}e^{-\frac{s}{2}\cdot d(p,\gamma (q))}$$
converges for $s>\delta(\Gamma)$ and any $p\in B^{4n+4}$, and diverges for any $s<\delta(\Gamma)$. For any convex cocompact subgroup $\Gamma$ of ${\rm
Sp}(n+1,1)$, there exists a probability measure $\mu_{\Gamma}$ supported on 	its limit set $\Lambda(\Gamma),$ called  the  \emph{Patterson-Sullivan
measure}, such that
\begin{align*}
\gamma^{*}\mu_{\Gamma}=|\gamma'|^{\delta(\Gamma)}\mu_{\Gamma}
\end{align*}
for any $\gamma\in\Gamma$ (cf. \cite{Corlette}), where $|\gamma'|$ is the conformal factor.

The conformal qc geometry of spherical qc manifolds can be applied to study convex cocompact subgroups of ${\rm Sp}(n+1,1).$ Let $G_{S}(\xi,\cdot)$ be the
Green function of the qc Yamabe operator   with the pole at $\xi$ on the standard sphere  $S^{4n+3}.$ In Section 6, by integrating Green function with
respect to the Patterson-Sullivan measure, we define a $C^{\infty}$ function on the open set  $\Omega(\Gamma):$
\begin{align*}
\phi_{\Gamma}(\xi)=\left(\int_{\Lambda(\Gamma)}
G_{S}^{\kappa}(\xi,\zeta)d\mu_{\Gamma}(\zeta)\right)^{\frac{1}{\kappa}},\quad
\kappa=\frac{2\delta(\Gamma)}{Q-2}.
\end{align*}
Then $$g_{\Gamma}=\phi_{\Gamma}^{\frac{4}{Q-2}}g$$ is invariant under $\Gamma,$ and so defines  a metric on $\Omega(\Gamma)/\Gamma$. This is the qc
generalization of Nayatani's canonical metric in conformal geometry \cite{Nayatani}.
We prove that if $\delta(\Gamma)<2n+2$ (resp. $\delta(\Gamma)>2n+2$, resp. $\delta(\Gamma)=2n+2$), then the scalar curvature of
$(\Omega(\Gamma)/\Gamma,{g}_{\Gamma},\mathbb{Q})$ is positive (resp. negative, resp. zero) everywhere. This result was proved for locally conformally flat
manifolds by Nayatani \cite{Nayatani}.
For spherical CR manifolds it was proved by Nayatani \cite{Naya} and Wang \cite{wang1} independently.

In the Appendix, we give a simple proof of the Green function of the qc Yamabe operator on the quaternionic Heisenberg group.

\section{Qc manifolds}
\subsection{Qc manifolds}
A  \emph{quaternionic contact (qc) manifold} $(M,g,\mathbb{Q})$ is a $(4n+3)$-dimensional manifold  $M$ with a codimension three distribution $H$
locally given as the kernel of a $\mathbb{R}^{3}$-valued $1$-form $\Theta=(\theta_{1},\theta_{2},\theta_{3})$, on which $g$ is a Carnot-Carath\'eodory
metric. In addition, $H$ has an {\it$ {\rm Sp}(n){\rm Sp}(1)$-structure}, that is, it is equipped with  a rank-three bundle
\begin{align*}
\mathbb{Q}=\{aI_{1}+bI_{2}+cI_{3}|a^{2}+b^{2}+c^{2}=1\},
\end{align*}
which consists of endomorphisms of $H$ locally generated by three almost complex structures $I_{1},I_{2},I_{3}$ on $H$ satisfying the commutating relation
of quaternions:
\begin{align}\label{2.00}
I_{1}I_{2}=-I_{2}I_{1}=I_{3},\ I_{1}^{2}=I_{2}^{2}=I_{3}^{2}=-id_{|H}.
\end{align}
They are hermitian
compatible with the metric:
\begin{align}\label{2.01}
g(I_{s}\cdot,I_{s}\cdot)=g(\cdot,\cdot),
\end{align} and satisfy  the  compatibility condition
\begin{align}\label{2,1}
g(I_{s}X,Y)=d\theta_{s}(X,Y),
\end{align}   for any  $X,Y\in H, s=1,2,3$. We denote $\mathbb{I}:=({I_{1},I_{2},I_{3}}).$

We say $(M,g,\mathbb{Q})$ is {\it conformal} to $(M,\tilde{g},\mathbb{Q})$ if $\tilde{g}=\phi^{\frac{4}{Q-2}}g$ for some smooth positive function $\phi$
on $M$. The conformal  class of qc manifolds is denoted by $(M,[g],\mathbb{Q}).$ In the definition of qc manifolds, the $\mathbb{R}^{3}$ valued $1$-form
$\Theta$ is unique up to a rotation by the following lemma.

\begin{lem}{\rm(cf. p. 100 in {\cite{Ivanov}})}\label{l2.1}
Let $(M,g,\mathbb{Q})$ be a qc manifold. If $\Theta$ and $\Theta'$ are two compatible $\mathbb{R}^{3}$-valued $1$-form such that ${\rm Ker}(\Theta)={\rm
Ker}(\Theta')=H$ locally, then we have  $\Theta'=\Psi\Theta$ for some $\Psi\in {\rm SO}(3)$-valued   smooth functions $\Psi$.
\end{lem}
\begin{proof}
Write $\Theta=(\theta_{1},\theta_{2},\theta_{3})$ and $\Theta'=(\theta'_{1},\theta'_{2},\theta'_{3})$. The condition ${\rm Ker}(\Theta) = {\rm
Ker}(\Theta')=H$ implies that $\theta'_{s}=\sum_{t=1}^{3}\Psi_{st}\theta_{t}$ for some ${\rm GL}(3)$-valued function $(\Psi_{st}).$ Applying the exterior
derivative, we find that $d\theta'_{s}=\sum_{t=1}^{3}(d\Psi_{st}\wedge\theta_{t}+\Psi_{st} d\theta_{t}),$ which restricted to $H$ gives
$g(I'_{s}X,Y)=\sum_{t=1}^{3}\Psi_{st}g(I_{t}X,Y)$ for any $X,Y\in H.$ Consequently we have
\begin{align}\label{321}
I'_{s}=\sum_{t=1}^{3}\Psi_{st}I_{t}.
\end{align}
So we must have $(\Psi_{st})\in {\rm SO}(3)$.
\end{proof}
Given a qc manifold $(M,{g},\mathbb{Q})$, there exists a canonical connection defined by Biquard
in \cite{biquard} when dim$M>7$, and by Duchemin in \cite{David} for the $7$-dimensional case. It is called \emph{Biquard connection} now.
\begin{thm}{\rm(cf. \S2.1.A. in \cite{biquard})}
If $(M^{4n+3},{g},\mathbb{Q})$ (for {\rm dim}$M>7$) has a qc structure  and $g$ is a Carnot-carath\'eodory metric on $H$, then there exists a
unique connection $\nabla$ on $H$ and a unique supplementary subspace $V$ of $H$ in
$TM$, such that\\
(i) $\nabla$ preserves the decomposition  $H\oplus V$ and the metric;\\
(ii) for $X,Y\in H$, one has $T_{X,Y}=-[X,Y]_{V}$;\\
(iii) $\nabla$ preserves the $ {\rm Sp}(n){\rm Sp}(1)$-structure on $H$;\\
(iv) for $R\in V$, the endomorphism $\cdot\rightarrow(T_{R,\cdot})_{H}$ of $H$
 lies in the orthogonal of $\mathfrak{sp}_{n}\oplus\mathfrak{sp}_{1}$;\\
(v) the connection on $V$ is induced by the natural identification of $V$
with the subspace $\mathfrak{sp}_{1}$ of the endomorphisms of $H$.
\end{thm}
A mapping $F:(\tilde{M},\tilde{g},\tilde{\mathbb{Q}})\rightarrow (M,g,\mathbb{Q})$ is called a \emph{conformal qc  mapping} if
\begin{align}\label{574}
F^{*}g=\phi\tilde{g},\quad  F^{*}\mathbb{Q}=\tilde{\mathbb{Q}},
\end{align}
for some positive function $\phi.$
The pull back $F^{*}A$ of an endomorphism  $A$ of $H$ is defined as
\begin{align}\label{322}
F^{*}A(X):=F_{*}^{-1}\left[A(F_{*}X)\right]
\end{align}
for any $X\in H.$
If $F$ is invertible and  $F$ and $F^{-1}$ are both qc mappings, $F$ is called a {\it qc diffeomorphism}.
If $f$ is a qc diffeomorphism, we must have $f_{*}\tilde{H}=H$, where $H$ and $\tilde{H}$ are horizontal subbundles of $TM$ and $T\tilde{M}$,
respectively.

\subsection{The quaternionic Heisenberg group}\label{s2.2}

The simplest  qc manifold is the \emph{quaternionic Heisenberg group} $\mathscr{H}^{n}=\mathbb{H}^{n}\oplus\rm{Im}\mathbb{H}$, whose multiplication is
given by
\begin{align*}
(y,t)\cdot(y',t')=(y+y',t+t'+2{\rm{Im}} ({y}\bar{y'})),
\end{align*}
where $y,{y'}\in \mathbb{H}^{n}$ and ${t}=t_{1}\mathbf{i}+t_{2}\mathbf{j}+t_{3}\mathbf{k},\
{t'}={t}=t'_{1}\mathbf{i}+t'_{2}\mathbf{j}+t'_{3}\mathbf{k}\in \rm{Im\mathbb{H}}$. The conjugation of a quaternion number
$x_{1}+x_{2}\mathbf{i}+x_{3}\mathbf{j}+x_{4}\mathbf{k}$ is $x_{1}-x_{2}\mathbf{i}-x_{3}\mathbf{j}-x_{4}\mathbf{k}.$ The neutral element is $(0,0)$ and the
inverse of $(y,{t})$ is $(-y,-{t})$.
The norm of the quaternionic Heisenberg group $\mathscr{H}^{n}$ is defined by
\begin{align}\label{124}
\|(y,{t})\|:=(|y|^{4}+|{t}|^{2})^{\frac{1}{4}}.
\end{align}

We have the following automorphisms of $\mathscr{H}^{n}$:\\
(1) \emph{dilations}:
\begin{align}
D_{\delta}:(y,{t})\longrightarrow(\delta y,\delta^{2}{t}),\ \delta>0;
\end{align}
(2) \emph{left translations}:
\begin{align}\label{2.25}
\tau_{(y',{t'})}:(y,{t})\longrightarrow
(y',{t'})\cdot(y,{t});
\end{align}
(3) \emph{rotations}:
\begin{align}\label{33}
{U}:(y,{t})\longrightarrow (yU,{t}),\ {\rm for} \ U\in {\rm Sp}(n),
\end{align}
where 
\begin{equation}\label{eq:Sp(n)}
  {\rm Sp}(n)=\{U\in {\rm GL}(n,\mathbb{H})|{U\bar{U}^{t}}=I_{n}\}; 
\end{equation}

(4) The \emph{inversion}:
\begin{align}\label{44}
R:(y,{t})\longrightarrow \left(-(|y|^{2}-t)^{-1}y,
\frac{-{t}}{|y|^{4}+|{t}|^{2}}\right);
\end{align}\\
(5) ${\rm Sp}(1)$ acts on $\mathscr{H}^{n}$ as:
\begin{align}\label{66}
\sigma:(y,{t})\longrightarrow (\sigma y,\sigma{t}\sigma^{-1}),
\end{align}
where the action on the first factor is left multiplication by $\sigma\in\mathbb{H}$ with $|\sigma|=1,$ while the action on the second factor  is
isomorphism with ${\rm SO}(3)$.

Note that for $x=x_{1}+x_{2}\textbf{i}+x_{3}\textbf{j}+x_{4}\textbf{k}$ and $x'=x'_{1}+x'_{2}\textbf{i}+x'_{3}\textbf{j}+x'_{4}\textbf{k},$ we have
\begin{equation}\label{584}
\begin{aligned}
{\rm Im}(x\bar{x'})&={\rm Im}\{(x_{1}+x_{2}\textbf{i}+x_{3}\textbf{j}+x_{4}\textbf{k})
(x'_{1}-x'_{2}\textbf{i}-x'_{3}\textbf{j}-x'_{4}\textbf{k})\}\\
&=(-x_{1}x'_{2}+x_{2}x'_{1}-x_{3}x'_{4}+x_{4}x'_{3})\textbf{i}
+(-x_{1}x'_{3}+x_{3}x'_{1}+x_{2}x'_{4}-x_{4}x'_{2})\textbf{j}\\
&\quad+(-x_{1}x'_{4}+x_{4}x'_{1}-x_{2}x'_{3}+x_{3}x'_{2})\textbf{k}
=:\sum_{k,j=1}^{4}b_{kj}^{s}x_{k}x'_{j}\textbf{i}_{s}
\end{aligned}
\end{equation}
(cf. (2.15) in \cite{wang2}), where $b_{kj}^{s}$ is the $(k,j)$-th entry of the following matrices $b^{s}:$
\begin{equation}\label{2.14}
\begin{aligned}
b^{1}:&=\left(\begin{array}{cccc} 0 & -1 & 0 &0\\ 1& 0& 0& 0\\ 0& 0&0& -1\\0 &0& 1 &0\end{array}\right), \quad   b^{2}:=\left(\begin{array}{cccc} 0 & 0 &
-1 &0\\ 0& 0& 0& 1\\ 1& 0&0& 0\\0 &-1& 0 &0\end{array}\right),\\
b^{3}:&=\left(\begin{array}{cccc} 0 & 0 & 0 &-1\\ 0& 0& -1& 0\\ 0& 1&0& 0\\1 &0& 0 &0\end{array}\right).
\end{aligned}
\end{equation}
It is easy to see that matrices $b^{1},b^{2},b^{3}$ satisfy the commutating relation of quaternions:
\begin{align}\label{2.144}
(b^{1})^{2}=(b^{2})^{2}=(b^{3})^{2}=-id,\qquad b^{1}b^{2}b^{3}=-id.
\end{align}
By (\ref{584}), the multiplication of the quaternionic Heisenberg group in terms of  real variables can  written as
 (cf.  \cite{wang2})
\begin{align*}
(y,{t})\cdot({y'},{t'})=\left(y+y',t_{s}+t_{s}'+2\sum_{l=0}^{n-1}
\sum_{j,k=1}^{4}
b_{kj}^{s}y_{4l+k}y'_{4l+j}\right),
\end{align*}
where $s=1,2,3,\ y=(y_{1},y_{2},\cdots,y_{4n}),\ y'=(y_{1}',y'_{2},\cdots,y'_{4n})\in \mathbb{R}^{4n},\ t=(t_{1},t_{2},t_{3}),\
t'=(t'_{1},t'_{2},t'_{3})\in \mathbb{R}^{3}.$
We denote
\begin{align}\label{2.43}
Y_{4l+j}:=\frac{\partial}{\partial y_{4l+j}}+2\sum_{s=1}^{3}\sum_{k=1}^{4}b^{s}_{kj}y_{4l+k}
\frac{\partial}{\partial t_{s}},
\end{align}
$l=0,\ldots, n-1,\ j=1,\ldots ,4.$ They are left invariant vector fields on the quaternionic Heisenberg group $\mathscr{H}^{n}$.

The horizontal subspace $H_{0}:=span\{Y_{1},\cdots,Y_{4n}\}$ generates the corresponding Lie algebra of the quaternionic Heisenberg group. The standard
$\mathbb{R}^{3}$-valued contact form of the group is
\begin{align}\label{654}
2\Theta_{0}:=dt-y\cdot d\bar{y}+dy\cdot\bar{y}.
\end{align}
If we write $ \Theta_{0}=(\theta_{0;1},\theta_{0;2},
\theta_{0;3}),$ then we have
\begin{align}\label{215}
2\theta_{0;s}=dt_{s}-2\sum_{l=0}^{n-1}\sum_{j,k=1}^{4}b_{kj}^{s}
y_{4l+k}dy_{4l+j},\quad s=1,2,3,
\end{align}
by using (\ref{584}) again. Then Ker$\Theta_{0}=H_{0}.$
The standard Carnot-Carath\'eodory metric on the group is defined as
\begin{align*}
g_{0}(Y_{\alpha},Y_{\beta})=2\delta_{\alpha\beta},
\end{align*}
for $\alpha,\beta =1,\cdots,4n.$  We set $\mathbb{Q}_{0}:=\{aI_{1}+bI_{2}+cI_{3}|a^{2}+b^{2}+c^{2}=1\},$ where transformations $I_{s},\ s=1,2,3,$ on
$H_{0}$ are given by
\begin{align*}
I_{s}Y_{4l+k}=\sum_{j=1}^{4}b_{jk}^{s}
Y_{4l+j},
\end{align*}
for $l=0,\cdots,n-1,\ k=1,2,3,4.$ It is direct to check $I_{1},I_{2}$ and $I_{3}$ satisfying the commutating relation of quaternions in (\ref{2.00}).
Recall that  the wedge product of  $1$-forms $\omega_{1}$ and $\omega_{2}$ is given by \begin{align}\label{345}
(\omega_{1}\wedge\omega_{2})(X,Y):
=\omega_{1}(X)\omega_{2}(Y)-
\omega_{1}(Y)\omega_{2}(X),
  \end{align}
for any vector field $X$ and $Y.$ It is easy to see that
\begin{align*}
d\theta_{0;s}(Y_{4l'+k},Y_{4l+j})
&=-\sum_{a=0}^{n-1}\sum_{j',k'=1}^{4}b_{k'j'}^{s}
dy_{4a+k'}\wedge dy_{4a+j'}(Y_{4l'+k},Y_{4l+j}) \\ &=-2b_{kj}^{s}\delta_{ll'}
=g_{0}(I_{s}Y_{4l'+k},Y_{4l+j}),
\end{align*}
since $b^{s}$ is antisymmetric.  Thus, $g_{0}$ is compatible with $\Theta_{0}.$ So $(\mathscr{H}^{n},g_{0},\mathbb{Q}_{0})$ is a qc manifold.

\subsection{The   quaternionic hyperbolic space and the standard qc structure on the sphere}
The {\it quaternionic projective space} $\mathbb{H}P^{n+1}$ of dimension $n+1$ is the set of left quaternionic lines in $\mathbb{H}^{n+2}$. More
precisely,
$$
\mathbb{H}P^{n+1}:=(\mathbb{H}^{n+2}\backslash\{0\})/\sim,
$$
where $\sim$ is the equivalent relation: $(p'_{1},\cdots,p'_{n+2})\sim(q'_{1},\cdots,q'_{n+2})\in \mathbb{H}^{n+2}$ if there is a non-zero quaternion
number $\lambda$ such that $$(p'_{1},\cdots,p'_{n+2})=( \lambda q'_{1},\cdots, \lambda q'_{n+2}).$$
Let $P:\mathbb{H}^{n+2}\backslash\{0\}\rightarrow \mathbb{H}P^{n+1}$ be the canonical projection onto the quaternionic projective space.
Under the induced action of ${\rm Sp}(n+1,1)$ on $\mathbb{H}P^{n+1}$, there are three invariant subsets
\begin{equation*}
\begin{aligned}
D_{+}:=\{q'\in\mathbb{H}P^{n+1};Q(q',q')>0\},\\
D_{0}:=\{q'\in\mathbb{H}P^{n+1};Q(q',q')=0\},\\
D_{-}:=\{q'\in\mathbb{H}P^{n+1};Q(q',q')<0\}.
\end{aligned}
\end{equation*}
Then, as a homogeneous space for  ${\rm Sp}(n+1,1)$,  $D_{+}$ is equivalent to the \emph{quaternionic hyperbolic space}.
In this case we must have $q_{n+2}\neq0.$  So a point in  $D_{+}$ is equivalent to $(q,1)$ for some $q\in\mathbb{H}^{n+1}$, i.e. $
({q'^{-1}_{n+2}}{q'_{1}},\cdots,{q'^{-1}_{n+2}}{q'_{n+1}},1).$

We introduce a positive definite hyperhermitian form on $\mathbb{H}^{n+1}$:
\begin{align*}
\langle q,p\rangle:={q}_{1}\bar{p}_{1}+\cdots+{q}_{n+1}\bar{p}_{n+1}.
\end{align*}
It is obvious that $U\in{\rm Sp(n+1)}$ (cf. (\ref{eq:Sp(n)}))  if and only if $\langle qU,pU\rangle=\langle q,p\rangle$ for any $q,p\in \mathbb{H}^{n+1}.$ We have the  ball model
for {quaternionic hyperbolic space}: \begin{align*}
B^{4n+4}=\left\{q\in\mathbb{H}^{n+1};\langle q,q\rangle<1\right\}.
\end{align*}

Let $\gamma=(\gamma_{ij})\in {\rm Sp}(n+1,1)$. $\gamma$ is a $(n+2)\times(n+2)$ matrix acts on left quaternionic vector space $\mathbb{H}^{n+2}$ from right. For
$q\in\mathbb{H}^{n+1},$ $(q,1)$ is a vector in $\mathbb{H}^{n+2}.$ The right action of $\gamma$ on this vector is denoted by $(q,1)\gamma,$ whose $l$-th
component is
$$[(q,1)\gamma]_{l}:=\sum_{m=1}^{n+1}q_{m}\gamma_{ml}+\gamma_{(n+2)l},\quad l=1,\cdots,n+2.$$
Note that $$([(q,1)\gamma]_{1},\cdots,[(q,1)\gamma]_{n+2})
\sim([(q,1)\gamma]_{n+2}^{-1}[(q,1)\gamma]_{1}
,\cdots,[(q,1)\gamma]_{n+2}^{-1}[(q,1)\gamma]_{n+1},1),$$ where $[(q,1)\gamma]_{n+2}\neq 0$ by $$|[(q,1)\gamma]_{1}|^{2}           +\cdots+
|[(q,1)\gamma]_{n+1}|^{2}-|[(q,1)\gamma]_{n+2}|^{2}=0,$$ for $q\in S^{4n+3}.$
So  ${\rm Sp}(n+1,1)$ induces an action on  $S^{4n+3}$ by
\begin{equation}
\begin{aligned}\label{3.16}
\gamma(q):=\left([(q,1)\gamma]^{-1}_{n+2}[(q,1)\gamma]_{1},\cdots,
[(q,1)\gamma]^{-1}_{n+2}[(q,1)\gamma]_{n+1}\right)\quad {\rm for}\ q\in S^{4n+3}.
\end{aligned}
\end{equation}
${\rm Sp}(n+1,1)$ also induces an action on $B^{4n+4}$ in this way.

The \emph{fundamental invariant} on the unit ball is given by
\begin{align}\label{580}
(q,p)=\frac{1-\langle q,p\rangle}{(1-|q|^{2})^{\frac{1}{2}}(1-|p|^{2})^{\frac{1}{2}}},
\end{align}
and
\begin{align}\label{2.20}
|(q,p)|=cosh\left(\frac{1}{2}d(q,p)\right),
\end{align}
for $q,p\in {B}^{4n+4},$ where $d(q,p)$ is the \emph{quaternionic hyperbolic distance} between $q$ and $p$, which is invariant under the action of ${\rm
Sp}(n+1,1)$ (cf. p. 523 in \cite{Corlette}).

For a $\mathbb{H}$-valued function $\mathbf{{f}}=f_{1}+f_{2}\mathbf{i}+f_{3}\mathbf{j}+f_{4}\mathbf{k}$, we set
\begin{align*}
d\mathbf{{f}}:&=\sum_{l=1}^{4n+4}\frac{\partial{\mathbf{f}}}
{\partial{x_{l}}}dx_{l}=
df_{1}+df_{2}\mathbf{i}+df_{3}\mathbf{j}+df_{4}\mathbf{k}.
\end{align*}
It is easy to see that for two $\mathbb{H}$-valued function $\textbf{f}$ and $\textbf{g},$ we have
\begin{align*}
d(\mathbf{f}\cdot\mathbf{g})&=d\mathbf{{f}}\cdot\mathbf{{g}}+\mathbf{{f}}
\cdot d\mathbf{{g}}.
\end{align*}
In particular, we have  $dq =dx_{0} +dx_{ 1}\mathbf{i}+dx_{ 2}\mathbf{j}+dx_{ 3}\mathbf{k}$ if we write  $q =x_{0} +x_{ 1}\mathbf{i}+x_{ 2}\mathbf{j}
+x_{ 3}\mathbf{k}$.
For a $\mathbb{H}^{n}$-valued function $\mathbf{f}=(\mathbf{f}_{1},\cdots,\mathbf{f}_{n})$, we write $d\mathbf{f}:=(d\mathbf{f}_{1},\cdots,d\mathbf{f}_{n})$.

For a point $\zeta=(\zeta_{1},\cdots,\zeta_{n+1})\in S^{4n+3}=\left\{\xi\in\mathbb{H}^{n+1}:|\xi|=1\right\}$, we consider a quaternionic subspace of the tangent space:
$$H_{\zeta}:=\{v\in \mathbb{H}^{n+1}:\langle v,{\zeta}\rangle=0\}.$$
It is easy to see that $H_{\zeta}$ is a left quaternionic subspace of ${\mathbb{H}}$-dimension $n.$ Then $H=\cup_{\zeta\in S^{4n+3}} H_{\zeta}$ is  the
standard horizonal bundle of the tangent bundle of  the sphere. Let $\eta:[0,1]\rightarrow S^{4n+3}$ be any smooth curve such that $\eta(0)=\zeta.$ We
identify the vector $(\eta_{1}'(0),\cdots,\eta_{n+1}'(0))$ with a tangential vector $X_{0}=\sum_{j=1}^{4n+4}v_{j}\frac{\partial}{\partial x_{j}}$ at point
$\zeta$, if we write  $\eta_{l}'(0)=v_{4l-3}+v_{4l-2}\mathbf{i}+v_{4l-1}\mathbf{j}+v_{4l}\mathbf{k}$.
The standard ${\rm Im}\mathbb{H}$-valued contact form on ${S}^{4n+3}$ is given by
\begin{align}\label{2.3}
\Theta_{S}
=\sum_{l=1}^{n+1}\left(d{\zeta}_{l}\cdot\bar{\zeta}_{l}
-{\zeta}_{l}\cdot d\bar{\zeta}_{l}\right),
\end{align}
where $\zeta=(\zeta_{1},\cdots,\zeta_{n+1})\in S^{4n+3}.$ If we write  $\zeta_{l}=x_{4l-3}+x_{4l-2}\textbf{i}+x_{4l-1}\textbf{j}+x_{4l}\textbf{k},\
l=1,\ldots,n+1,$ and $\Theta_{S}=(\theta_{1},\theta_{2},\theta_{3}),$ then we have
\begin{align*}
\theta_{s}=-2\sum_{l=0}^{n}\sum_{k,j=1}^{4}b_{kj}^{s}x_{4l+k}dx_{4l+j},
\end{align*}
by using identity (\ref{584}) again, and so
\begin{align*}
d\theta_{s}=-2\sum_{l=0}^{n}\sum_{k,j=1}^{4}b_{kj}^{s}dx_{4l+k}\wedge dx_{4l+j}.
\end{align*}
The transformation $I_{s}$ on $H_{\zeta}$ is given by left multiplying $\mathbf{i}_{s}:$
\begin{align*}
\left(\eta'_{1}(0),\cdots,\eta'_{n+1}(0)\right)\longmapsto
\left(\textbf{i}_{s}\eta'_{1}(0),\cdots,\textbf{i}_{s}\eta'_{n+1}(0)\right).
\end{align*}
We can check that
\begin{equation}\begin{aligned}
\mathbf{i}_{s}\left(x_{1}+x_{2}\textbf{i}+x_{3}\textbf{j}+x_{4}\textbf{k}\right)
=\sum_{j,k=1}^{4}b^{s}_{jk}x_{k}\mathbf{i}_{j-1}
\end{aligned}\end{equation}
(here $\mathbf{i}_{0}=1$) (cf. (2.2) in \cite{wang5}). We define
\begin{align*}
I_{s}\partial_{x_{4l+k}}:=\sum_{j=1}^{4}b_{jk}^{s}
\partial_{x_{4l+j}},
\end{align*}
for $l=0,\cdots,n-1,\ j=1,2,3,4,\ s=1,2,3.$
Then for $v\in H_{\zeta},$ we have $\langle I_{s}v,\zeta\rangle=0,$ i.e. $I_{s}v\in H_{\zeta}.$
$\mathbb{Q}_{s}=\{aI_{1}+bI_{2}+cI_{3}|a^{2}+b^{2}+c^{2}=1\}$ is a ${\rm Sp}(n){\rm Sp}(1)$-structure on $H.$
Denote by $g_{S}$ the restriction to the horizontal subspace $H$ of the Euclidean metric on $S^{4n+3} $ multiplying a factor $4$. We can prove that
$g_{S}$ is compatible to $\Theta_{S}$ and $\mathbb{Q}_{S}$ on $S^{4n+3} $ as in the case of  the quaternionic Heisenberg group,  i.e. for any $X,Y\in H,$ $$g_{S}(I_{s}X,Y)=d\theta_{s}(X,Y),\quad s=1,2,3.$$

The group of  conformal qc transformations of $S^{4n+3}$ consists of  quaternionic fractional linear transformation ${\rm Sp}(n+1,1)/{\rm center}$
\cite{Ivanov4}.  We have the following qc Liouville type theorem.

\begin{thm}\label{liou}{{\rm (\textbf{Qc Liouville type theorem}) (cf. Theorem 8.5 in \cite{Ivanov4}) }}
Every conformal qc transformation between open subsets of $S^{4n+3}$ is the restriction
of a global conformal qc transformation.
\end{thm}
We can identify $\mathscr{H}^{n}$ with the boundary $\Sigma$ of the \emph{Siegel domain} in $\mathbb{{H}}^{n+1}$,
\begin{align}\label{Sigma}
\Sigma:=\{({q},q_{n+1})\in\mathbb{{H}}^{n}\times\mathbb{{H}}: {\rm{Re}}{q_{n+1}}=|{q}|^{2}\},
\end{align}
by using the projection
\begin{equation*}
\begin{aligned}
\hat{\pi}:\qquad\Sigma\quad&\longrightarrow\quad\mathscr{H}^{n}, \\
({q},q_{n+1})&\longmapsto({q},|{q}|^{2}-q_{n+1}).
\end{aligned}\end{equation*}

The \emph{Cayley transform} is the map from the sphere $S^{4n+3}$ minus the southern point to the quadratic hypersurface $\Sigma$ defined by
\begin{align*}
\hat{F}:S^{4n+3}\longrightarrow\Sigma, \quad(\zeta,\zeta_{n+1})\longmapsto(q,q_{n+1}),
\end{align*}
where
\begin{align}\label{237}
q=(1+\zeta_{n+1})^{-1}\zeta,\quad q_{n+1}=(1+\zeta_{n+1})^{-1}(1-\zeta_{n+1}).
\end{align}
Then we have the \emph{stereographic projection}:
\begin{align}\label{2.24}
F=\hat{\pi}\circ\hat{F}:S^{4n+3}\longrightarrow \mathscr{H}^{n}, \quad(\zeta,\zeta_{n+1})\longmapsto(q,t),
\end{align}
given by
\begin{align*}
q=(1+\zeta_{n+1})^{-1}\zeta,\quad t=\frac{2{\it \texttt{Im}}\zeta_{n+1}}{|1+\zeta_{n+1}|^{2}}.
\end{align*}
\begin{prop}{\rm (cf. p.146 in \cite{Ivanov})}\label{p2.1}
The Cayley transform is a conformal qc diffeomorphism
between  $\mathscr{H}^{n}$ with its standard qc contact structure ${\Theta}_{0}$  in (\ref{654}) and the sphere minus a point with its standard qc contact
structure $\Theta_{S}$ in (\ref{2.3}). More precisely,
\begin{align}\label{225}
{F}^{*}{\Theta}_{0}=
{\alpha}\frac{\Theta_{S}}{2|1+\zeta_{n+1}|^{2}}
\bar{\alpha} ,
\end{align}
where $\alpha=\frac{\overline{1+\zeta_{n+1}}}{|1+\zeta_{n+1}|}$ is a unit quaternion.
\end{prop}
 We have the following corollary.
\begin{cor}\label{c2.1}
Suppose that $g_{S}$ is the standard qc metric on $S^{4n+3}$ and $g_{0}$ is the standard qc metric on $\mathscr{H}^{n}.$ We have
\begin{align}\label{227}
{F}^{*}{g}_{0}=\frac{g_{S}}{2|1+\zeta_{n+1}|^{2}},
\end{align}
where $F$ is the stereographic projection defined by (\ref{2.24}).
\end{cor}
\begin{proof}
Let $\mathbb{I}=(I_{1},I_{2},I_{3})$ be the standard qc structure on the group $\mathscr{H}^{n}$ and $\mathbb{I}'=(I_{1}',I_{2}',I_{3}')$ be the standard
qc structure on the sphere. Let $\omega_{s}:=g_{0}(I_{s}\cdot,\cdot)=d\theta_{s},$ if we write $\Theta_{0}=(\theta_{1},\theta_{2},\theta_{3}).$ Similarly,
let $\omega'_{s}:=g_{S}(I'_{s}\cdot,\cdot)=d\theta'_{s},$ if we write $\Theta_{S}=(\theta'_{1},\theta'_{2},\theta'_{3}).$ Consider the fundamental $4$-form
\begin{align*}
\Omega=\omega_{1}\wedge\omega_{1}+\omega_{2}\wedge\omega_{2}+
\omega_{3}\wedge\omega_{3},
\end{align*}
on the horizontal subspace $H_{0;p}$ for a fixe point $p\in \mathscr{H}^{n}.$
It is known that an element  $g\in {\rm GL}(4n,\mathbb{R})$ preserving $\Omega$ if and only if $g\in {\rm Sp}(n){\rm Sp}(1)$ (cf. Lemma 9.1 in
\cite{Salamon}), where ${\rm Sp}(n)=\{A\in {\rm O}_{g_0}(4n);AI_{s}=I_{s}A,s=1,2,3\}$ and ${\rm
Sp}(1)=\{a_{1}I_{1}+a_{2}I_{2}+a_{3}I_{3};a_{1}^{2}+a_{2}^{2}+a_{3}^{2}
=1\}.$ Here $A\in {\rm O}_{g_0}(4n)$ means that $A$ is orthogonal with respect to $g_{0}$ on $H_{0;p}.$ Similar result holds for $\Omega'.$

Write $\alpha\Theta_{S}\bar{\alpha}=(\cdots,\sum_{k=1}^{3}a_{jk}\theta'_{k},\cdots).$ It follows from (\ref{225}) and direct calculation that $F^{*}\omega_{j}=\lambda\sum_{k=1}^{3}a_{jk}\omega_{k}'$ as $2$-forms on the horizontal subspace, where $\lambda=\frac{1}{2|1+\xi_{n+1}|^{2}} $, and so
\begin{align*}
F^{*}\Omega=\lambda^{2}\Omega'.
\end{align*}
Then   $A$ preserving
$\Omega $ implies $F^{*}A$ preserving $\Omega'$. This is because
\begin{equation*}
\begin{aligned}
\Omega'\left(F^{*}A(X_{1}),\cdots,F^{*}A(X_{4})\right)&=\lambda^{-2}F^{*}\Omega
\left(F^{-1}_{*}[A(F_{*}X_{1})],\cdots,F^{-1}_{*}[A(F_{*}X_{4})]\right)\\&=
\lambda^{-2}\Omega
\left(A(F_{*}X_{1}),\cdots,A(F_{*}X_{4})\right)\\&=\lambda^{-2}\Omega
\left(F_{*}X_{1},\cdots,F_{*}X_{4}\right)=\Omega'(X_{1},\cdots,X_{4}),
\end{aligned}\end{equation*}
where $X_{1},\cdots,X_{4}\in H_{0;p}.$ Thus $F^{*}{A}$ is orthogonal with respect to $g_S$ if $A$ is orthogonal with respect to $g_0$, and vise versa. But for any $X,Y\in H_{S,p}$ with
$\|F_{*}X\|_{g_{0}}=\|F_{*}Y\|_{g_{0}}$,  there exists an $A\in {\rm O}_{g_{0}}(H),$  such that $F_{*}Y=AF_{*}X.$ Then we have
\begin{equation*}
\begin{aligned}
\|Y\|_{g_S}=\|F^{-1}_{*}(F_{*}Y)\|_{g_S}=\|F^{-1}_{*}[A(F_{*}X)]\|_{g_S}
=\|F^{*}A(X)\|_{g_S}=\|X\|_{g_S}.
\end{aligned}
\end{equation*}
This implies that $F^{*}g_{0}$ is conformal to $g_{S},$ i.e. we can write $F^{*}{g_{0}}=\mu g_{S},$ for some $\mu>0.$
Consequently, we have $\left.F^{*}{\rm Vol}\right|_{H_{0}}=\mu^{2n}\left. {\rm Vol}\right|_{H_{S}}.$ 

On the other hand, $F^{*}d\theta_{s}|_{H_S}=\lambda\alpha
d\theta'_{s}\bar{\alpha}|_{H_S}=\lambda g_{S}(\hat{I}_{s}\cdot,\cdot)|_{H_S}=\lambda\hat{\omega}_{s}|_{H_S}$ for $s=1,2,3,$ where $\hat{\mathbb{I}} =\alpha \mathbb{I}' \bar{\alpha},$ and
$\hat{I}_{s}$ is also an almost complex structure compatible with $g_{S}.$
It follows that $\left(F^{*}d\theta_{s}\right)^{2n}=\lambda^{2n}\hat{\omega}^{2n}_s $ on ${H_S}$,  and so $\left.F^{*}{\rm Vol}\right|_{H_{0}}=\lambda^{2n}\left. {\rm
Vol}\right|_{H_S} $. So we must have $\mu=\lambda.$
The corollary is proved.
\end{proof}
Then by the transformation formula (\ref{227}), we get the following corollary.
\begin{cor}\label{333} For any $\xi=(q,t)\in\mathscr{H}^{n}$, we have
\begin{align}\label{22}
\left.R^{*}g_{0}\right|_{\xi}=\left.\frac{1}{\|\xi\|^{4}}
g_{0}\right|_{\xi},\quad D_{r}^{*}g_{0}=r^{2}g_{0}.
\end{align}
\end{cor}
\begin{proof}
Note that the Cayley transformation $\hat{F}$ maps $(-\zeta,-\zeta_{n+1})$ to
\begin{align}\label{242}
\left(-(1-\zeta_{n+1})^{-1}\zeta,
(1-\zeta_{n+1})^{-1}(1+\zeta_{n+1})\right)
\end{align}
for $(\zeta,\zeta_{n+1})\in S^{4n+3}.$ The reflection $R$ in (\ref{44}) on the Heisenberg group induces a reflection $\hat{R}:=\hat{\pi}^{-1}\circ
R\circ\hat{\pi}$ on the quadratic surface  $\Sigma$ in (\ref{Sigma}). It is direct to see that
\begin{align*}
\hat{F}(-\zeta,-\zeta_{n+1})=\left(-q_{n+1}^{-1}q,q_{n+1}^{-1}\right)
=\hat{R}(q,q_{n+1}),
\end{align*}
for $(q,q_{n+1})=\hat{F}(\zeta,\zeta_{n+1})$ by (\ref{237}). Then $\hat{F}\circ \varphi\circ \hat{F}^{-1}=\hat{R}$, where $\varphi$ is the qc isometric
automorphism of $S^{4n+3}$ given by $(\zeta,\zeta_{n+1})\rightarrow (-\zeta,-\zeta_{n+1}).$ Consequently,  $F\circ \varphi\circ F^{-1}=R.$
Then, by Corollary \ref{c2.1}, we have
$$\left.R^{*}g_{0}\right|_{\xi}=\left.\left(F^{-1}\right)^{*}\circ \varphi^{*}\circ
F^{*}g_{0}\right|_{\xi}=\left.\frac{|1+\zeta_{n+1}|^{2}}{|1-\zeta_{n+1}|^{2}}
g_{0}\right|_{\xi}
=\left.\frac{1}{|q_{n+1}|^{2}}g_{0}|_{\xi}=\frac{1}{\|\xi\|^{4}}
g_{0}\right|_{\xi},$$
by (\ref{237}), where $\xi=(\zeta,\zeta_{n+1})$.
The result follows.
\end{proof}
\subsection{The conformal action of ${\rm Sp}(n+1,1)$ on the sphere} Let us show the elements of ${\rm Sp}(n+1,1)$ acting  conformally on the sphere with the standard qc structure. We need to know the conformal factor explicitly.
\begin{prop}\label{p2.2}
Suppose that  $\gamma \in {\rm Sp}(n+1,1)$ and $\Theta_{S}$ is the standard contact form on $S^{4n+3}$. Then we have
\begin{align}\label{111}
\left.\gamma^{*}\Theta_{S}\right|_{\zeta}={\lambda}\frac{\Theta_{{S}}}{|
[(\zeta,1)\gamma]_{n+2}|^{2}}\bar{\lambda}
\end{align}
at point $\zeta\in S^{4n+3}$, where \begin{align}\label{5.11}\lambda=\frac{\overline{{[(\zeta,1)\gamma]_{n+2}}}}
{\left|[(\zeta,1)\gamma]_{n+2}\right|}
\end{align} is a unit quaternion.
\end{prop}
\begin{proof}
Note that for any $x\in \mathbb{H}$,
\begin{align}\label{2.40}
d(x^{-1})=-\bar{x}\frac{dx}{|x|^{4}}\bar{x}.
\end{align}
By the action of ${\rm Sp}(n+1,1)$ on the ball $\bar{B}^{4n+4}$ in (\ref{3.16}), we have
\begin{align*}
d[(\zeta,1)\gamma]_{l}=\sum_{m=1}^{n+1}d\zeta_{m}\gamma_{ml}
=[(d\zeta,0)\gamma]_{l},
\end{align*} for $\zeta\in \mathbb{H}^{n+1}$, 
$ l=1,\cdots,n+1,$ where $d\zeta=(d\zeta_{1},\cdots,d\zeta_{n+1}).$ We differentiate (\ref{3.16}) to get
\begin{align*}
d\gamma(\zeta)=\left(\cdots,-\frac{
\overline{[(\zeta,1)\gamma]_{n+2}}[(d\zeta,0)\gamma]_{n+2}
\overline{[(\zeta,1)\gamma]_{n+2}}}
{|[(\zeta,1)\gamma]_{n+2}|^{4}}[(\zeta,1)\gamma]_{l}+
[(\zeta,1)\gamma]^{-1}_{n+2}[(d\zeta,0)\gamma]_{l},\cdots\right).
\end{align*}
We find that $\overline{\gamma({\zeta})_{l}}
=\overline{[(\zeta,1)\gamma]_{l}}\cdot
\overline{[(\zeta,1)\gamma]^{-1}_{n+2}}$ and
\begin{equation}\label{2.42}
\begin{aligned}
\quad&\sum_{l=1}^{n+1}{d\gamma(\zeta)_{l}}\overline{\gamma({\zeta})_{l}}\\
=&-\frac{\overline{[(\zeta,1)\gamma]_{n+2}}
{[(d\zeta,0)\gamma]_{n+2}}
}{|[(\zeta,1)\gamma]_{n+2}|^{4}}\sum_{l=1}^{n+1}|[(\zeta,1)\gamma]_{l}|^{2}
+{[(\zeta,1)\gamma]^{-1}_{n+2}}\sum_{l=1}^{n+1}
{[(d\zeta,0)\gamma]_{l}}
\overline{[(\zeta,1)\gamma]_{l}}\cdot\overline{[(\zeta,1)\gamma]^{-1}_{n+2}}\\
=&-\frac{\overline{[(\zeta,1)\gamma]_{n+2}}{[(d\zeta,0)\gamma]_{n+2}}}
{|[(\zeta,1)\gamma]_{n+2}|^{2}}
+{[(\zeta,1)\gamma]^{-1}_{n+2}}\sum_{l=1}^{n+1}
{[(d\zeta,0)\gamma]_{l}}\overline{[(\zeta,1)\gamma]_{l}}\cdot
\overline{[(\zeta,1)\gamma]^{-1}_{n+2}}\\
=&{\lambda}\frac{\sum_{l=1}^{n+1}d{\zeta}_{l}
\bar{\zeta}_{l}}{|[(\zeta,1)\gamma]_{n+2}|^{2}}\bar{\lambda},
\end{aligned}
\end{equation}
by $\gamma$ preserving the hyperhermitian form $Q(\cdot,\cdot)$ in (\ref{2.31}), i.e.,
\begin{align}\label{2.36}
-\sum_{l=1}^{n+1}{[(\zeta,1)\gamma]_{l}}\overline{[(\zeta',1)\gamma]_{l}}
+{[(\zeta,1)\gamma]_{n+2}}\overline{[(\zeta',1)\gamma]_{n+2}}
=-{\zeta}\bar{\zeta'}+1
\end{align}
with $\zeta=\zeta'$ in  the third identity and $|\zeta|^{2}=1$. The last identity follows from differentiating (\ref{2.36}) with respect to $\zeta$ and
setting $\zeta'=\zeta$ to get
$$\sum_{l=1}^{n+1}{[(d\zeta,0)\gamma]_{l}}\overline{[(\zeta,1)\gamma]_{l}}-
{[(d\zeta,0)\gamma]_{n+2}}\overline{[(\zeta,1)\gamma]_{n+2}}=
\sum_{l=1}^{n+1}d{\zeta}_{l}\bar{\zeta}_{l}.$$
Then (\ref{2.42}) minus its conjugate gives
\begin{equation*}
\begin{aligned}
\quad&\sum_{l=1}^{n+2}\left(d{\gamma(\zeta)}_{l}\overline{\gamma({\zeta})_{l}}-
{\gamma(\zeta)}_{l}d\overline{\gamma({\zeta})_{l}}\right)={\lambda}\frac{\Theta_{S}}{|
[(\zeta,1)\gamma]_{n+2}|^{2}}\bar{\lambda}.
\end{aligned}
\end{equation*}
The proposition is proved.
\end{proof}

Let $X_{0}$ be the tangential vector given by differentiating along  a curve $\eta(t)=\left(\eta_{1}(t),\cdots,\eta_{n+1}(t)\right)$ at point $\zeta.$
Then the tangential  vector $\gamma_{*}X_{0}$  at point $\gamma(\zeta)$ is given by differentiating along the curve $\gamma(\eta(t)),$ i.e.
$\left(\left.\frac{d}{dt}\right|_{t=0}\gamma(\eta(t))_{1}
,\cdots,\\\left.\frac{d}{dt}\right|_{t=0}\gamma(\eta(t))_{n+1}\right)$,
where
\begin{equation}\label{251}
\begin{aligned} \left.\frac{d}{dt}\right|_{t=0}\gamma(\eta(t))_{l}
={[(\zeta,1)\gamma]_{n+2}^{-1}}[(\eta'(0),0)\gamma]_{l}
-A_{\zeta}\cdot[(\zeta,1)\gamma]_{l}
\end{aligned}
\end{equation}
by the definition (\ref{3.16}). Here
\begin{align}\label{2277}
A_{\zeta}=\left.\frac{d}{dt}\right|_{t=0}[(\eta(t),1)\gamma]_{n+2}^{-1}
=\frac{\overline{[(\zeta,1)\gamma]_{n+2}}
[(\eta'(0),0)\gamma]_{n+2}
\overline{[(\zeta,1)\gamma]_{n+2}}}
{|[(\zeta,1)\gamma]_{n+2}|^{4}}\in \mathbb{H}
\end{align}
by (\ref{2.40}).

The following proposition gives us the conformal factor of the transformation of ${\rm Sp}(n+1,1)$ on the sphere with the standard qc metric.

\begin{prop}\label{1234}
For $\gamma \in {\rm Sp}(n+1,1)$, we have
\begin{align}\label{252}
H_{\gamma(\zeta)}=\gamma_{*}H_{\zeta} \quad {\rm and}\quad
\gamma^{*}g_{S}=\frac{1}{|
[(\zeta,1)\gamma]_{n+2}|^{2}}{g_{S}}.
\end{align}
\end{prop}
\begin{proof}
Let $\eta$ be a curve in $S^{4n+3}$ with $\eta(0)=\zeta$ such that $\eta'(0) $ is horizontal. By $\gamma$ preserving the hyperhermitian form in
(\ref{2.36}) again, we have
\begin{align}\label{253}
-\sum_{l=1}^{n+1}{[(\eta(t),1)\gamma]_{l}}\overline{[(\eta(s),1)\gamma]_{l}}+
{[(\eta(t),1)\gamma]_{n+2}}\overline{[(\eta(s),1)\gamma]_{n+2}}=-
\langle\eta(t),\eta(s)\rangle+1.
\end{align}
Differentiate it with respect to $t$ at $0$ and then let $s\rightarrow 0$ to get
\begin{align}\label{254}
-\sum_{l=1}^{n+1}
{[(\eta'(0),0)\gamma]_{l}}\overline{[(\zeta,1)\gamma]_{l}}+
{[(\eta'(0),0)\gamma]_{n+2}}\overline{[(\zeta,1)\gamma]_{n+2}}
=-\langle\eta'(0),{\zeta}\rangle=0
\end{align}
by $\eta'(0)\in H_{\zeta}.$   Then we have
\begin{equation*}
\begin{aligned}
&\left\langle\left.
\frac{d\gamma(\eta(t))}{dt}\right|_{t=0},\gamma(\zeta)\right\rangle=\sum_{l=1}^{n+1}\left({{[(\zeta,1)\gamma]_{n+2}^{-1}}}
\cdot{[(\eta'(0),0)\gamma]_{l}}-
{A_{\zeta}}\cdot{[(\zeta,1)\gamma]_{l}}
\right)\overline{{[(\zeta,1)\gamma]_{l}}}\cdot
\overline{[(\zeta,1)\gamma]^{-1}_{n+2}}=0,
\end{aligned}
\end{equation*}
by (\ref{251}), (\ref{254}) and
\begin{equation}\label{878}\begin{aligned}
{A_{\zeta}}\sum_{l=1}^{n+1}{[(\zeta,1)\gamma]_{l}}
\overline{[(\zeta,1)\gamma]_{l}}\cdot
\overline{[(\zeta,1)\gamma]^{-1}_{n+2}}={A_{\zeta}}\cdot {[(\zeta,1)\gamma]_{n+2}} ={[(\zeta,1)\gamma]_{n+2}^{-1}} \cdot{[(\eta'(0),0)\gamma]_{n+2}}
\end{aligned}\end{equation}
by (\ref{2.36}) and (\ref{2277}). It implies that $\left.\frac{d}{dt}\right|_{t=0}\gamma(\eta(t))\in H_{\gamma(\zeta)}.$ Thus $\gamma_{*}H_{\zeta}\subset
H_{\gamma(\zeta)}.$ They actually coincide since $\gamma$ is invertible. Denote $\|X\|^{2}:=g_{S}(X,X)$ for $X\in H.$
Moreover, by (\ref{251}) and (\ref{254}) again we have
\begin{equation*}
\begin{aligned}
\frac{1}{4}\|\gamma_{*}X_{0}\|^{2}&=
\sum^{n+1}_{l=1}\left|\left.\frac{d}{dt}\right|_{t=0}
\gamma(\eta(t))_{l}\right|^{2}
=\sum_{l=1}^{n+1}\left|
[(\zeta,1)\gamma]_{n+2}^{-1}[(\eta'(0),0)\gamma]_{l}\right|^{2}\\
&\quad+\sum_{l=1}^{n+1}\left|A_{\zeta}[(\zeta,1)\gamma]_{l}\right|^{2}
-2{\rm Re}\sum_{l=1}^{n+1}
{[(\zeta,1)\gamma]_{n+2}^{-1}[(\eta'(0),0)\gamma]_{l}}
\cdot{\overline{[(\zeta,1)\gamma]_{l}}\cdot\overline{A_{\zeta}}}\\
&=\frac{\sum_{l=1}^{n+1}|[(\eta'(0),0)\gamma]_{l}|^{2}}
{|[(\zeta,1)\gamma]_{n+2}|^{2}}+
|A_{\zeta}|^{2}\sum_{l=1}^{n+1}|[(\zeta,1)\gamma]_{l}|^{2}-
2\frac{|[(\eta'(0),0)\gamma]_{n+2}|^{2}}
{|[(\zeta,1)\gamma]_{n+2}|^{2}} =\frac{\sum_{l=1}^{n+1}|\eta_{l}'(0)|^{2}}
{|[(\zeta,1)\gamma]_{n+2}|^{2}}.
\end{aligned}
\end{equation*}
Here we have used 
$$\sum_{l=1}^{n+1}\left|[(\eta'(0),0)\gamma]_{l}\right|^{2}
-\left|[(\eta'(0),0)\gamma]_{n+2}\right|^{2}
=\sum_{l=1}^{n+1}\left|\eta_{l}'(0)\right|^{2},$$
which follows from differentiating (\ref{253}) with respect to $s$ and $t$ at $0.$
Then by  $\gamma^{*}g_{S}(X_{0},X_{0})=g_{S}(\gamma_{*}X_{0},\gamma_{*}X_{0})$, we complete the proof of (\ref{252}).
\end{proof}
Now we give the pull back formula of the quaternionic structure $\mathbb{I}$ in the following proposition.
\begin{cor}\label{p2.5}
For $\gamma \in {\rm Sp}(n+1,1)$, we have
\begin{align*}
\gamma^{*}\mathbb{I}=\lambda{\mathbb{I}}\bar{\lambda},
\end{align*}
where $\lambda$ is a unit quaternion given by (\ref{5.11}).
\end{cor}
\begin{proof}
It follows from (\ref{321}) in the proof of  Proposition \ref{p2.2}.
\end{proof}
\subsection{Spherical qc manifolds and connected sums}\label{s2.5}
A local sphere theorem was proved by Ivanov and Vassilev in  \cite{Ivanov2}, i.e. a qc manifold is spherical if and only if its conformal qc curvature of
the Biquard connection vanishes. See also \cite{Alt} for a proof using parabolic geometry.
The following proposition tells us that this definition coincides with the topological definition gives at the beginning of this paper.
\begin{prop}\label{p2.4}
A qc manifold $(M,g,\mathbb{Q})$ is  {spherical} if and only if  it is a manifold with coordinate charts $\{(U_{i},\phi_{i})\}$, where
$\phi_{i}:U_{i}\rightarrow S^{4n+3}$  and transition  maps are given by induced action of elements of ${\rm Sp}(n+1,1).$
\end{prop}
\begin{proof}
The necessity follows from the qc Liouville type Theorem \ref{liou}.

Given such coordinate charts $\{(U_{i},\phi_{i})\}$ of $M,$ let us construct a qc metric $g$ and a bundle $\mathbb{Q}$ on $M.$  Let $\chi_{i}$ be a unit
partition subordinating to the cover $\{U_{i}\}$, i.e. ${\rm supp}\chi_{i}\subset U_{i}$ and $\sum_{i}\chi_{i}\equiv1.$ Let
$g:=\sum_{k}\chi_{k}\phi_{k}^{*}g_{S}$ be a Carnot-Carath\'eodory metric on $H_{p}=\phi_{i}^{*}H_{\phi_{i}(p)},$ where $H_{\phi_{i}(p)}$ is the horizontal
subspace of the sphere at the point $\phi_{i}(p)$. This definition of $H_{p}$ is independent of  the choice of $i,$ since
$(\phi_{ji})_{*}H_{\phi_{i}(p)}=H_{\phi_{j}(p)}$ by Proposition \ref{1234},  where transition map $\phi_{ji}:=\phi_{j}\circ\phi_{i}^{-1}\in {\rm
Sp}(n+1,1).$  Then on the open set $\phi_{i}(U_{i})\subset S^{4n+3},$ we have
\begin{align*}
\left(\phi_{i}^{-1}\right)^{*}g=\left(\phi_{i}^{-1}\right)^{*}
\left(\sum_{k}\chi_{k}\phi_{k}^{*}g_{S}\right)
=\chi_{i}\circ\phi_{i}^{-1}\cdot g_{S}+\sum_{k\neq i}\chi_{k}\circ\phi_{i}^{-1}\cdot \phi_{ki}^{*}g_{S}
=Bg_{S}
\end{align*}
for some positive function $B$ on $\phi_{i}(U_{i})$ by
the pull  back formula (\ref{252}) for the metric $g_{S}$ in Proposition \ref{1234}.  Therefore $g$ is a spherical qc metric on $M$. By Proposition
\ref{p2.5}, we have
\begin{equation*}
\begin{aligned}
\phi_{ji}^{*}\mathbb{I} =\nu\mathbb{I}\bar{\nu}\quad {\rm with}\  \nu(\zeta)=\frac{\overline{[(\zeta,1)\phi_{ji}]_{n+2}}}
{|[(\zeta,1)\phi_{ji}]_{n+2}|},
\end{aligned}
\end{equation*}
on $U_{i}\cap U_{j}.$ We have $\phi_{ii}^{*}(\mathbb{I})=\mathbb{I}$ and $\phi_{kj}^{*}\circ \phi_{ji}^{*}(\mathbb{I})=\phi_{ki}^{*}(\mathbb{I}).$ Namely
$\phi_{ji}^{*}$ satisfy the cocycle  condition and define a sphere bundle.  We can choose  $\Theta=\phi_{i}^{*}(B\Theta_{S})$ locally, and the
compatibility  condition (\ref{2.01})-(\ref{2,1}) for $g$ and $\Theta$ obviously holds, since $d\Theta=\phi_{i}^{*}(Bd\Theta_{S})$ when restricted to the
horizontal subspace.
\end{proof}
For $\xi\in\mathscr{H}^{n}$ and $\epsilon>0,$ define a {\it ball on the quaternionic Heisenberg group} as
\begin{align*}
B_{\mathscr{H}}(\xi,\epsilon):=\{\eta\in\mathscr{H}^{n}
;\|\xi^{-1}\cdot\eta\|<\epsilon\}.
\end{align*}
Let $(M,g,\mathbb{Q})$ be a spherical qc manifold of dimension $4n+3$ with two punctures $\eta_{1},\eta_{2}$, or disjoint union of two connected spherical
qc manifolds $(M_{(1)},g_{(1)},\mathbb{Q}_{(1)}),$ $(M_{(2)},g_{(2)},\mathbb{Q}_{(2)})$ with one puncture $\eta_{i}\in M_{(i)},i=1,2$, each. Let $U_{1}$
and $U_{2}$ be two disjoint neighborhoods of $\eta_{1}$ and $\eta_{2}$, respectively. Let
\begin{align}\label{265}
\psi_{i}:U_{i}\rightarrow B_{\mathscr{H}}(0,2),\quad i=1,2,
\end{align}
be local coordinate charts such that $\psi_{i}(\eta_{i})=0$.
For $t<1$, define
\begin{equation*}
\begin{aligned}
U_{i}(t,1):&=\{\eta\in U_{i};t<\|\psi_{i}(\eta)\|<1\},\\
U_{i}(t):&=\{\eta\in U_{i};\|\psi_{i}(\eta)\|<t\},
\end{aligned}
\end{equation*}
$i=1,2.$ For any $t\in(0,1), A\in {\rm Sp}(n)$ and $\sigma \in {\rm Sp}(1)$, we can form a new spherical qc manifold $M_{t,\sigma,A}$ by removing the
closed balls $\overline{U_{i}(t)},i=1,2,$ and gluing  $U_{1}(t,1)$ with $U_{2}(t,1)$ by the conformal qc  mapping $\Psi_{t,\sigma,A}:U_{1}(t,1)\rightarrow
U_{2}(t,1)$ defined by
\begin{align}\label{2.60}
\Psi_{t,\sigma,A}(\eta)=\psi_{2}^{-1}\circ D_{t}\circ R\circ\sigma\circ {A}\circ\psi_{1}(\eta), \ {\rm for} \ \eta\in U_{1}(t,1),
\end{align}
where $R:\{\zeta\in \mathscr{H}^{n};t<\|\zeta\|<1\}\rightarrow\{\zeta\in \mathscr{H}^{n};1<\|\zeta\|<\frac{1}{t}\}$  is the inversion  in  (\ref{44}). Note that $\Psi_{t,\sigma,A}$ is conformal qc with respect to the spherical qc structure
$\mathbb{Q}$ on $M$, since $D_{t}\circ R\circ\sigma\circ U_{A}$ is a qc automorphism of the quaternionic Heisenberg group.  Such a conformal qc
transformation maps $U_{1}(t,1)$ to $U_{2}(t,1),$ which identifies the inner boundary of $U_{1}(t,1)$ with the outer boundary $U_{2}(t,1)$ and vise versa.
Let
\begin{align}\label{54}
\pi_{t,\sigma,A}:(M_{1}\backslash \overline{U_{1}(t)})\cup (M_{2}\backslash \overline{U_{2}(t)})\rightarrow M_{t,\sigma,A}
\end{align}
be a canonical projection.  We call $M_{t,\sigma,A}$ the \emph{connected sum} of $M_{1}$ and $M_{2}$, which is a spherical  qc manifold by Proposition
\ref{p2.4}. We denote this spherical qc manifold by $(M_{t,\sigma,A},g,\mathbb{Q}_{t,\sigma,A})$, where $g$ is a metric in the conformal class given by
Proposition \ref{p2.4}. As in the locally conformally case, the connected sums are expected to be not isomorphic for some different choices of
$t,\sigma,A$ \cite{Izeki1}. As in the locally conformally flat case, it is interesting to investigate their moduli space.
\section{The qc Yamabe operator and its Green function}
\subsection{The qc Yamabe operator}
On a qc manifold $(M,g,\mathbb{Q})$, let us  choose a local basis  $\{X_{j}\}_{j=1}^{4n}$ of the horizontal subspace $H$  as in \cite{wang2}.
We choose a local section $e_{1}$ of $H$ such that $g(e_{1},e_{1})=1.$ Then, $e_{1},I_{1}e_{1},I_{2}e_{1},I_{3}e_{1}$ are mutually orthonormal. Set
$I_{0}=id_{|H}$. Now choose a local section $e_{2}$ of $H$ orthonormal to span$\{I_{k}e_{1}|k=0,\cdots,3\}.$ Then $e_{2},I_{1}e_{2},I_{2}e_{2},I_{3}e_{2}$
are mutually orthonormal again and span$\{I_{k}e_{1}|k=0,\cdots,3\}$ $\bot$ span$\{I_{k}e_{2}|k=0,\cdots,3\}$. Repeating this procedure, we can find
$e_{1},\cdots,e_{n}$ such that $\{I_{k}e_{j}|j=1,\cdots,n,\ k=0,\cdots,3\}$ is a local orthonormal basis of $H$ under the metric $g.$ Set
\begin{align*}
X_{4l+\alpha+1}:=\sqrt{2}I_{\alpha }e_{l+1}
\end{align*}
for $l=0,\cdots,n-1,\ \alpha=0,\cdots,3.$ Then $$g(X_{j},X_{k})=2\delta_{jk}.$$
The Carnot-Carath\'eodory metric $g$ induces a dual metric on $H^{*}$, denoted  by $\langle\cdot,\cdot\rangle_{g}$. Since the Biquard connection $\nabla$
preserves $H$, we can write the covariant derivative as
$\nabla_{X_{j}}X_{k}=\Gamma_{jk}^{k'}X_{k'}.$
By  definition, $(\nabla_{X}\omega)Y=\nabla_{X}(\omega(Y))-\omega(\nabla_{X}Y)$ for a $1$-form $\omega\in\Omega^{1}(M).$
Then  we   define an $L^{2}$ inner product $\langle \cdot,\cdot\rangle_{g,\mathbb{Q}}$ on $\Gamma(H^{*})$ by
\begin{align*}
\langle \omega,\omega'\rangle_{g,\mathbb{Q}}:=\int_{M}\langle \omega,\omega'\rangle_{g}dV_{g,\mathbb{Q}},
\end{align*}
where the volume form $dV_{{g,\mathbb{Q}}}$  is
\begin{align}\label{320}
dV_{g,\mathbb{Q}}:={\theta}_{1}\wedge{\theta}_{2}\wedge
{\theta}_{3}\wedge(d{\theta}_{s})^{2n},\end{align}
$ s=1,2,3, $ if we choose $\Theta=({\theta}_{1},{\theta}_{2},{\theta}_{3})$ locally.
The  volume form is independent of $s$ and the choice of $\Theta$ by the follow proposition.

\begin{prop}\label{234}
The volume element $dV_{{g,\mathbb{Q}}}$  only depends on $g$ and $\mathbb{Q}$, not on $s$ or the choice of the $\mathbb{R}^{3}$-valued contact form
$\Theta=(\theta_{1},\theta_{2},\theta_{3}).$
\end{prop}
\begin{proof}
Let  $1$-forms $\{\theta^j\}_{j=1}^{4n}$ be  the basis dual to $\{X_{j}\}_{j=1}^{4n}.$   Recall the structure equation (3.6) in \cite{wang2}, i.e.
\begin{align*}
d\theta_{s}=-\sum_{l=0}^{n-1}\sum_{j,k=1}^{4}b_{kj}^{s}\theta^{4l+k}
\wedge\theta^{4l+j},\quad {\rm mod}\  \theta_{1},\theta_{2},\theta_{3},  \end{align*}
$s=1,2,3,$ where $b^{s}$ is  given by (\ref{2.14}).
By using $b^{1}$ in (\ref{2.14}), it is direct to check that
\begin{align*}
(d{\theta}_{1})^{2n}&=\left(\sum_{l=0}^{n-1}2\theta^{4l+1}\wedge\theta^{4l+2}
+2\theta^{4l+3}\wedge\theta^{4l+4}\right)^{2n}
 =2^{2n}(2n)!\theta^{1}\wedge\cdots\wedge\theta^{4n},\quad {\rm mod}\  \theta_{1},\theta_{2},\theta_{3}.
\end{align*}
Similarly,
\begin{equation}
\begin{aligned}\label{3.7}
(d{\theta}_{s})^{2n}=
2^{2n}(2n)!\theta^{1}\wedge\cdots\wedge\theta^{4n},\quad {\rm mod}\  \theta_{1},\theta_{2},\theta_{3},
\end{aligned}
\end{equation}
$s=2,3,$ and so $dV_{g,\mathbb{Q}}$ is independent of $s$.
Let
$\tilde{\Theta}=(\tilde{\theta}_{1},\tilde{\theta}_{2},\tilde{\theta}_{3})$ be another contact form satisfying $d\tilde{\theta}_{s}(X,Y)=g(I_{s}X,Y).$ We
can write $\tilde{\theta}_{s}=\sum_{j=1}^{3}c_{sj}\theta_{j},s=1,2,3,$ for some ${\rm SO}(3)$-valued function $(c_{ij})$ by Lemma \ref{l2.1} and
simultaneously $\tilde{I}_{s}=\sum_{j=1}^{3}c_{sj}I_{j},s=1,2,3,$ by the proof of Lemma \ref{l2.1}. As the procedure at the beginning of this section, we can choose
orthonormal basis $\{\cdots,e_{k},\tilde{I}_{1}e_{k},\tilde{I}_{2}e_{k},
\tilde{I}_{3}e_{k},\cdots\}.$ Then we have
\begin{align*}
\tilde{X}_{4l+1}=X_{4l+1},\quad \tilde{X}_{4l+\alpha+1}=\sqrt{2}\tilde{I}_{\alpha}e_{l+1}=\sum_{\beta=1}^{3} c_{\alpha\beta}X_{4l+\beta+1}
\end{align*}
for $l=0,\cdots,n-1,\ \alpha=1,2,3.$
Dually, we have the dual basis $\{\tilde{\theta}^{i}\}$ such that
\begin{align*}
\tilde{\theta}^{4l+1}=\theta^{4l+1},\quad
\tilde{\theta}^{4l+\alpha+1}=\sum_{\beta=1}^{3}c_{\alpha\beta}\theta^{4l+\beta+1},
\end{align*}
for $l=0,\cdots,n-1,\ \alpha=1,2,3.$ In fact, we have $\tilde{\theta}^{4l+1}\wedge\cdots\wedge\tilde{\theta}^{4l+4}=\det(c_{ij})
{\theta}^{4l+1}\wedge\cdots\wedge{\theta}^{4l+4}$ and $\det(c_{ij})=1$ by $(c_{ij})\in {\rm SO}(3).$ By (\ref{3.7}), we have
\begin{equation}\label{3.77}
\begin{aligned}
dV_{\tilde{g},\mathbb{Q}}&=\tilde{\theta}_{1}\wedge\tilde{\theta}_{2}
\wedge\tilde{\theta}_{3}\wedge\left(d\tilde{\theta}_{s}
\right)^{2n}=2^{2n}(2n)!\det(c_{ij})\theta_{1}\wedge\theta_{2}\wedge
\theta_{3}\wedge\tilde{\theta}^{1}\wedge\cdots\wedge\tilde{\theta}^{4n}   \\&=2^{2n}(2n)!{\theta}_{1}\wedge{\theta}_{2}\wedge
{\theta}_{3}\wedge\theta^{1}\wedge\cdots\wedge\theta^{4n}=
{\theta}_{1}\wedge{\theta}_{2}\wedge
{\theta}_{3}\wedge(d{\theta}_{s})^{2n}.
\end{aligned}
\end{equation}
  The proposition is proved.
\end{proof}
Denote $d_{b}:={\rm pr}\circ d,$ where ${\rm pr}$ is the projection from $T^{*}M$ to $H^{*}$. We define the \emph{SubLaplacian} $\Delta_{g,\mathbb{Q}}$
associated to the qc contact structure $(M,g,\mathbb{Q})$ by
\begin{align}\label{3.9}
\int_{M}\Delta_{g,\mathbb{Q}}u\cdot vdV_{g,\mathbb{Q}}=
\int_{M}\langle d_{b}u,d_{b}v\rangle_{g }dV_{g,\mathbb{Q}}
\end{align}
for $u,v\in C_{0}^{\infty}(M).$ The SubLaplacian $\Delta_{g,\mathbb{Q}}$ has the following expression.
\begin{prop}\label{2.4}{\rm(cf. p.365 in \cite{wang2})}
For $u\in C^{\infty}(M)$, we have
\begin{align}\label{231}
\Delta_{g,\mathbb{Q}}u=\frac{1}{2}\sum_{j=1}^{4n}\left(-X_{j}X_{j}u+\sum_{k=1}
^{4n}\Gamma_{kk}^{j}X_{j}u\right).
\end{align}
\end{prop}

The transformation law of the scalar curvature under the conformal changes was given by Biquard in \cite{biquard} for dim$M>7.$ When dim$M=7$, it was
given by Duchemin in \cite{David}.
\begin{thm}\label{1.5}
Under the conformal change $\tilde{g}=f^{2}g$,  the scalar curvature becomes
\begin{align*}
s_{\tilde{g},\mathbb{Q}}=f^{-2}\left(s_{g,\mathbb{Q}}
-8(n+2)Tr^{H}\nabla\alpha-16(n+1)(n+2)|\alpha|^{2}\right)
\end{align*}
where $\alpha=f^{-1}df$ and $\nabla$ is the Biquard connection.
\end{thm}
Writing the conformal factor $f$ as $e^{h}$, we can write the transformation law in the following form.
\begin{cor}\label{2.32}
The scalar curvature $s_{\tilde{g},\mathbb{Q}}$ associated with the structure
$\tilde{g}=e^{2h}g$ satisfies
\begin{align*}
s_{{\tilde{g},\mathbb{Q}}}=e^{-2h}\left(s_{g,\mathbb{Q}}+ 2(Q+2)\Delta_{g,\mathbb{Q}}h
-\sum_{j=1}^{4n}\frac{(Q+2)(Q-2)}{2}(X_{j}h)^{2}\right)
\end{align*}
\end{cor}
\begin{proof}
Since $f=e^{h}$, $\alpha=dh$. Then,
\begin{equation*}
\begin{aligned}
Tr^{H}\nabla\alpha&=\sum_{j=1}^{4n}(\nabla_{{X_{j}}/{\sqrt{2}}}\alpha)
({X_{j}}/{\sqrt{2}})
=\frac{1}{2}\sum_{j}^{4n}(\nabla_{X_{j}}dh)(X_{j})
\\&=\frac{1}{2}\sum_{j}^{4n}\left(X_{j}(dh(X_{j}))-
dh(\nabla_{X_{j}}X_{j})\right)=\frac{1}{2}\sum_{j=1}^{4n}\left(X_{j}X_{j}h-\sum_{k=1}
^{4n}\Gamma_{kk}^{j}X_{j}h\right)=-\Delta_{g,\mathbb{Q}}h.
\end{aligned}
\end{equation*}
The result follows.
\end{proof}
\begin{cor}{\rm(cf. p. 360 in \cite{wang2})}\label{1.6}
The scalar curvature $s_{\tilde{g},\mathbb{Q}}$ associated with the structure $(M,\tilde{g},\mathbb{Q})$ with
$\tilde{g}=\phi^{\frac{4}{Q-2}}g$ satisfies qc Yamabe equation:

\begin{align}\label{122}
b_{n}\Delta_{g,\mathbb{Q}}\phi+s_{g,\mathbb{Q}}\phi=
s_{\tilde{g},\mathbb{Q}}\phi^{\frac{Q+2}{Q-2}},\quad b_{n}=4\frac{Q+2}{Q-2}.
\end{align}
\end{cor}

Now let us derive the transformation law of the qc Yamabe operator. See \cite{Orsted} for such derivation in the pseudoRiemannian case and \cite{Wang3} in
  the CR case.
\begin{prop}
Let $({M},\tilde{g},\mathbb{Q})$ and $(M,g,\mathbb{Q})$ be two  qc manifolds. Let $\tilde{g}=\phi^{\frac{4}{Q-2}}g$ for some positive smooth function
$\phi$ on ${M}$. Then
\begin{align}\label{123}
  \Delta_{g,\mathbb{Q}}(\phi\cdot f)=\Delta_{g,\mathbb{Q}}\phi\cdot f+{\phi^{\frac{Q+2}{Q-2}}}\Delta_{\tilde{g},\mathbb{Q}}f
\end{align}
for any smooth function $f$ on $M$.
\end{prop}

\begin{proof}
Let $\Theta=(\theta_{1},\theta_{2},\theta_{3})$ be a $\mathbb{R}^{3}$-valued $1$-form associated to $({M},{g},\mathbb{Q})$ and let
$\tilde{\Theta}=(\tilde{\theta}_{1},\tilde{\theta}_{2},\tilde{\theta}_{3})$ be associated to $({M},\tilde{{g}},\mathbb{Q})$.
For any real function $h$ on $M$, we have
\begin{equation}\label{11}
\begin{aligned}
\langle\Delta_{g,\mathbb{Q}}(\phi\cdot f),h\rangle_{g,\mathbb{Q}}
&=\int_{M}\langle d_{b}(\phi\cdot f),d_{b}h\rangle_{g}dV_{g,\mathbb{Q}}
=\langle d_{b}\phi\cdot f+\phi d_{b} f,d_{b}h\rangle_{g,\mathbb{Q}}\\
&=\langle d_{b}\phi,f\cdot d_{b}h\rangle_{g,\mathbb{Q}}+
\langle d_{b} f,\phi\cdot d_{b}h\rangle_{g,\mathbb{Q}}\\
&=\langle d_{b}\phi,d_{b}(f\cdot h)\rangle_{g,\mathbb{Q}}
+\langle d_{b}f,\phi\cdot d_{b}h- h\cdot d_{b}\phi\rangle_{g,\mathbb{Q}}\\
&=\langle \Delta_{g,\mathbb{Q}}\phi,f\cdot h\rangle_{g,\mathbb{Q}}
+\langle d_{b}f,\phi\cdot d_{b}h- h\cdot d_{b}\phi\rangle_{g,\mathbb{Q}}.
\end{aligned}
\end{equation}

Let us calculate the second term in the right side of (\ref{11}). By our assumption and Proposition \ref{234}, we just need to consider
$\tilde{\Theta}=\phi^{\frac{4}{Q-2}}\Theta.$ Then for a fixed $s$, we have
\begin{align}\label{12}
d\tilde{\theta}_{s}=d(\phi^{\frac{4}{Q-2}}\theta_{s})
={\frac{4}{Q-2}}\phi^{\frac{6-Q}{Q-2}}d\phi\wedge\theta_{s}
+\phi^{\frac{4}{Q-2}}d\theta_{s},
\end{align}
i.e. $d\tilde{\theta}_{s}=\phi^{\frac{4}{Q-2}}d\theta_{s}\  {\rm mod}\ \theta_{1},\theta_{2},\theta_{3}.$ So, we get
\begin{align}\label{3.19}
dV_{\tilde{g},\mathbb{Q}}=\tilde{\theta}_{s}\wedge\tilde{\theta}_{2}
\wedge\tilde{\theta}_{3}\wedge
(d\tilde{\theta}_{s})^{2n}=\phi^{\frac{2Q}{Q-2}}
{\theta}_{1}\wedge{\theta}_{2}\wedge{\theta}_{3}\wedge
(d{\theta}_{1})^{2n}=\phi^{\frac{2Q}{Q-2}}dV_{g,\mathbb{Q}}.
\end{align}
Consequently, for $1$-forms $\omega_{1},\omega_{2}\in \Gamma(H^{*}),$ we have
\begin{equation*}
\begin{aligned}
\langle \omega_{1},\omega_{2}\rangle_{g,\mathbb{Q}}=
\int_{M} \langle \omega_{1},\omega_{2}\rangle_{g}dV_{g,\mathbb{Q}}
=\langle\phi^{-2}\omega_{1},\omega_{2}\rangle_{\tilde{g},\mathbb{Q}}.
\end{aligned}
\end{equation*}
Now we find that
\begin{equation*}
\begin{aligned}
\langle d_{b}f,\phi \cdot d_{b} h- h\cdot d_{b}\phi\rangle_{g,\mathbb{Q}}
&=\langle\phi^{-2} d_{b}f,\phi\cdot d_{b} h- h\cdot d_{b}\phi\rangle_{\tilde{g},\mathbb{Q}}
=\langle d_{b}f,d_{b}(\phi^{-1}h)\rangle_{\tilde{g},\mathbb{Q}}\\
&=\int \Delta_{\tilde{g},\mathbb{Q}}f\cdot\phi^{-1}hd V_{\tilde{g},\mathbb{Q}}
=\int \phi^{\frac{Q+2}{Q-2}}\Delta_{\tilde{g},\mathbb{Q}}f\cdot hd V_{g,\mathbb{Q}}.
\end{aligned}
\end{equation*}
The proposition is proved.
\end{proof}

\begin{cor}
The   qc Yamabe operator $L_{g,\mathbb{Q}}$ satisfies the transformation law
\begin{align}\label{121}
L_{\tilde{g},\mathbb{Q}}f=\phi^{-\frac{Q+2}{Q-2}}L_{g,\mathbb{Q}}(\phi f),
\end{align}
if $\tilde{g}=\phi^{\frac{4}{Q-2}}g$ and $f\in C^{\infty}(M)$.
\end{cor}

\begin{proof}

By using (\ref{122}) and (\ref{123}), we have
\begin{equation*}
\begin{aligned}
L_{g,\mathbb{Q}}(\phi  f)&=b_{n}\Delta_{g,\mathbb{Q}}(\phi f)+s_{g,\mathbb{Q}}\phi f=b_{n}\left(\Delta_{g,\mathbb{Q}}\phi\cdot f+\phi^{\frac{Q+2}{Q-2}}
 \Delta_{\tilde{g},\mathbb{Q}}f\right)
+s_{g,\mathbb{Q}}\phi f\\
&=\phi^{\frac{Q+2}{Q-2}}\left(b_{n}\Delta_{\tilde{g},\mathbb{Q}}f
+s_{\tilde{g},\mathbb{Q}}f\right).
\end{aligned}
\end{equation*}
The result follows.
\end{proof}

\subsection{The Green function of the qc Yamabe operator}
For simplicity we write $G_{g_{0},\mathbb{Q}},$ $\Delta_{g_{0},\mathbb{Q}},$ $L_{g_{0},\mathbb{Q}},$ $\langle\cdot,\cdot\rangle_{g_{0},\mathbb{Q}}$ as
$G_{0},$ $\Delta_{0},$ $L_{0},$ $\langle\cdot,\cdot\rangle_{0}$. Since  the Biquard connection on $\mathscr{H}^{n}$ is trivial \cite{Ivanov}, its
connections and curvatures vanish. For $u\in C^{1}(\mathscr{H}^{n}),$ we have
\begin{align*}
du=\sum_{j=1}^{4n}Y_{j}u\cdot\theta^{j}+\sum_{s=1}
^{3}\frac{\partial u}{\partial t_{s}}\cdot\theta_{0;s}
\end{align*}
and $d_{b}u=\sum_{j=1}^{4n}Y_{j}u\cdot\theta^{j}$, where $\theta^{j}=d y_{j},$ and $Y_{j}$ is given by (\ref{2.43}). Recall that $\langle
Y_{j},Y_{k}\rangle_{{0}}=2\delta_{jk}$ and $\langle\theta^{j},\theta^{k}\rangle_{{0}}=\frac{1}{2}\delta_{jk}$ for $j,k=1,\cdots,4n.$ Hence,
\begin{align*}
\langle d_{b}u,d_{b}v\rangle_{0}=\frac{1}{2}\sum_{j=1}^{4n}
Y_{j}u\cdot Y_{j}v
\end{align*}
for real functions $u,$  $v,$ and the SubLaplacian is
\begin{align*}
\Delta_{0}=-\frac{1}{2}\sum_{j=1}^{4n}Y_{j}Y_{j}
\end{align*}
by (\ref{231}).

A continuous function $G_{g,\mathbb{Q}}:M\times M\backslash {\rm diag}M\rightarrow \mathbb{R}$ is called the \emph{Green function} of the qc Yamabe
operator  $L_{g,\mathbb{Q}}$ if
\begin{align*}
\int_{M}G_{g,\mathbb{Q}}(\xi,\eta)L_{g,\mathbb{Q}}u(\eta)
dV_{g,\mathbb{Q}}(\eta)=u(\xi)
\end{align*}
for all $u\in C_{0}^{\infty}(M).$ Namely, $L_{g,\mathbb{Q}}G_{g,\mathbb{Q}}(\xi,\cdot)=\delta_{\xi}.$

In \cite{kaplan} Kaplan found the explicit form of the fundamental solution of the SubLaplacian  on groups of H-type. Since the  quaternionic Heisenberg
groups is a group of H-type, we know its fundamental solution.
\begin{prop}\label{1.10}
The Green function of the qc Yamabe operator $L_{0}=b_{n}\Delta_{0}$ on the quaternionic Heisenberg group $\mathscr{H}^{n}$ with the pole at $\xi$ is
\begin{align*}
G_{0}(\xi,\eta):=\frac{C_{Q}}{\|\xi^{-1}\eta\|^{Q-2}},
\end{align*}
for $\xi\neq \eta,\ \xi,\eta\in \mathscr{H}^{n}$, where $\|\cdot\|$ is the norm on $\mathscr{H}^{n}$ defined by (\ref{124}) and
\begin{align}\label{146}
{C_{Q}}^{-1}=8(n+1)(n+2)b_{n}
\int_{\mathbb{R}^{4n+3}}\frac{|y|^{2}}{(|y|^{4}+|t|^{2}+1)^{n+3}}dV_{0},
\end{align}
\noindent where $dV_{0}$ is \emph{Lebesgue measure}.
\end{prop}
In the Appendix,  we prove this proposition by  simple and direct calculation.

\begin{prop}\label{2.1}
Let $(M,g,\mathbb{Q})$ be a connected compact qc manifold. Then one and only one of the following cases holds: there is  a qc metric $\tilde{g}$ conformal
to $g$  which have  either positive,   negative or  vanishing scalar curvature  everywhere.
\end{prop}

\begin{proof}
The qc Yamabe operator $L_{g,\mathbb{Q}}$ is a formally
self-adjoint and subelliptic differential operator. So its spectrum is real and bounded from below. Let $\lambda_{1}$ be the first eigenvalue of
$L_{g,\mathbb{Q}}$ and let $\phi$ be an eigenfunction of $L_{g,\mathbb{Q}}$ with eigenvalue $\lambda_{1}$. Then $\phi>0$ and is  $C^{\infty}$ by Theorem
3.6 in \cite{Ivanov4}. The scalar curvature of $(M,\tilde{g},\mathbb{Q})$ with $\tilde{g}=\phi^{\frac{4}{Q-2}}g$ is
$s_{\tilde{g},\mathbb{Q}}=\lambda_{1}\phi^{-\frac{4}{Q-2}}$ by the qc Yamabe equation (\ref{122}). In particular, $s_{\tilde{g},\mathbb{Q}}>0$ (resp.
$s_{\tilde{g},\mathbb{Q}}<0,$  resp. $s_{\tilde{g},\mathbb{Q}}\equiv0$) if $\lambda_{1}>0$  (resp. $\lambda_{1}<0$, resp. $\lambda_{1}=0$).
On the other hand, if $\hat{g}$ has scalar curvature $s_{\hat{g},\mathbb{Q}}>0$ (resp. $s_{\hat{g},\mathbb{Q}}<0,$ resp. $s_{\hat{g},\mathbb{Q}}\equiv0$),
the first eigenvalue ${\hat{\lambda}}_{1}$ of $L_{\hat{g},\mathbb{Q}}$ obviously satisfies ${\hat{\lambda}}_{1}>0$  (resp. ${\hat{\lambda}}_{1}<0$, resp.
${\hat{\lambda}}_{1}=0$).
\end{proof}
\begin{rem} This proposition does not exclude the possibility that a scalar positive qc manifold has a metric  with  scalar curvature vanishing or negative
somewhere.
\end{rem}
Define the qc \emph{Yamabe invariant}
\begin{align}\label{4.7}
\lambda(M,g,\mathbb{Q}):=\inf_{u>0}
\frac{\int_{M}\left(b_{n}|\nabla_{g}u|^{2}+s_{g,\mathbb{Q}}
u^{2}\right)dV_{g,\mathbb{Q}}}{\left(\int_{M}u^{\frac{2Q}{Q-2}}
dV_{g,\mathbb{Q}}\right)^{\frac{Q-2}{Q}}},
\end{align}
where $|\nabla_{g}f|^{2}=\sum_{j=1}^{4n}|X_{j}f|^{2}$ if $\{X_{j}\}$ is a local orthogonal basis of $H$ under the Carnot-Carath\'eodory metric $g.$ It is
an invariant  for the conformal  class of qc manifolds (cf. p. 361 in \cite{wang2}).

Theorem 1.1 in \cite{wang2}  tells us that, if $\lambda(M,g,\mathbb{Q})<\lambda(\mathscr{H}^{n},g_{0},\mathbb{Q}_{0})$, then (\ref{4.7}) can be achieved
by a positive $C^{\infty}$ solution $u$ of the qc Yamabe equation  (\ref{122}), i.e. it has a constant scalar curvature
$s_{\tilde{g},\mathbb{Q}}=\lambda(M,g,\mathbb{Q})$ for $\tilde{g}=u^{\frac{4}{Q-2}}g$. It is known that  $\lambda(\mathscr{H}^{n},g_{0},\mathbb{Q}_{0})>0$
(cf. Corollary 2.1 \cite{Ivanov5}). Thus, $(M,g,\mathbb{Q})$ is scalar negative or zero if and only if $\lambda(M,g,\mathbb{Q})$ is negative or zero,
respectively. Consequently,  $(M,g,\mathbb{Q})$ is scalar positive if and only if $\lambda(M,g,\mathbb{Q})$ is positive.

From now on in this section, we assume $(M,g,\mathbb{Q})$ to be connected,  compact and scalar positive.  In this case, the qc Yamabe operator
$L_{g,\mathbb{Q}}=b_{n}\Delta_{g,\mathbb{Q}}+s_{g,\mathbb{Q}}$ is a positive operator and its inverse  always exists. The Green function is the Schwarz
kernel of the inverse operator and can be constructed as in the following proposition. Moreover, we also find the singular part of the Green function. In
this case, the Green function of the qc Yamabe operator $L_{g,\mathbb{Q}}$ is unique.

\begin{prop}\label{p3.7}
Let $(M,g,\mathbb{Q})$ be a connected  compact spherical qc manifold with positive scalar curvature and let $U$ be a sufficiently small open set. Then the
function $G_{g,\mathbb{Q}}(\xi,\eta)-\rho_{g,\mathbb{Q}}(\xi,\eta)$ can be extended to a $C^{\infty}$ function on $U\times U$, where
$\rho_{g,\mathbb{Q}}(\cdot,\cdot)$ is given by (\ref{222}).
\end{prop}

\begin{proof} The proof is similar to the CR case. 
Suppose that $\bar{U}\subset \tilde{U}\subset \mathscr{H}^{n}$ and $g=\phi^{\frac{4}{Q-2}}g_{0}$ on $\tilde{U}$. We choose a sufficiently small $\rho$
such that $B_{\mathscr{H}}(\xi,\rho)\subset\tilde{U}$ for any $\xi\in U$. We can construct the Green function as follows. For $\xi, \eta\in U$, define
\begin{align*}
\tilde{G}(\xi, \eta)=\tilde{G}(\xi^{-1}\eta),
\end{align*}
where $\tilde{G}$ is the cut-off fundamental solution, i.e.  $\tilde{G}(\tilde{\eta})=\frac{C_{Q}}{\|\tilde{\eta}\|^{Q-2}}
f(\tilde{\eta})$  for $\tilde{\eta}\in \mathscr{H}^{n}.$ Here $f\in C_{0}^{\infty}(\mathscr{H}^{n})$  satisfying $f\equiv 1 $ on
$B_{\mathscr{H}}(0,\frac{\rho}{2})$ and $f\equiv 0$ on $B_{\mathscr{H}}(0,{\rho})^{c}$.
Recall that $L_{0}=-\frac{b_{n}}{2}\sum_{j=1}^{4n}Y_{j}Y_{j},$ where $Y_{j}$ are given by (\ref{2.43}).
Then,
\begin{equation}\label{3.30}
\begin{aligned}
L_{0}\tilde{G}(\tilde{\eta})=&L_{0}\left(\frac{C_{Q}}{\|\tilde{\eta}\|^{Q-2}}
f(\tilde{\eta})\right)
=\delta_{0}-{b_{n}}\sum_{j=1}^{4n}Y_{j}
\left(\frac{C_{Q}}{\|\tilde{\eta}\|^{Q-2}}\right)Y_{j}f(\tilde{\eta})+\frac{C_{Q}}
{\|\tilde{\eta}\|^{Q-2}}
L_{0}f(\tilde{\eta})\\=:
&\delta_{0}+\tilde{G}_{1}(\tilde{\eta})
\end{aligned}
\end{equation}
by $\frac{C_{Q}}{\|\tilde{\eta}\|^{Q-2}}$ being the fundamental solution of $L_{0}$ and  $Y_{j}f\equiv0$ on $B_{\mathscr{H}}(0,\frac{\rho}{2})$. Here
$\delta_{0}$ is the Dirac function at the origin with respect to  the measure $dV_{0}$ and $\tilde{G}_{1}$ is defined by the last equality in
(\ref{3.30}).
Set ${G}_{1}(\xi,\eta):=\tilde{G}_{1}(\xi^{-1}\eta)$ for $\xi,\eta\in U$. Then, ${G}_{1}(\xi,\eta)\in C^{\infty}(U\times U)$ and for each $\xi\in U$,
${G}_{1}(\xi,\cdot)$ can be naturally extended to a smooth function on $M$ satisfying ${G}_{1}(\xi,\eta)=0$ for $\eta\notin\tilde{U}$. By
transformation law (\ref{121}) and left invariance of $Y_{j}$, we find that
\begin{equation*}
\begin{aligned}
L_{g,\mathbb{Q}}(\phi(\xi)^{-1}\phi(\cdot)^{-1}{\tilde{G}}(\xi,\cdot))
&=\phi(\xi)^{-1}\phi(\cdot)^{-\frac{Q+2}{Q-2}}L_{0}{\tilde{G}}(\xi,\cdot)
=\phi(\xi)^{-1}\phi(\cdot)^{-\frac{Q+2}{Q-2}}(\delta_{0}(\xi^{-1}\cdot)+{G}_{1}(\xi,\cdot))\\
&=\delta_{\xi}+\phi(\xi)^{-1}\phi(\cdot)^{-\frac{Q+2}{Q-2}}{G}_{1}(\xi,\cdot),
\end{aligned}
\end{equation*}
on $U$ for $\xi\in U$ , where $\delta_{\xi}$ is the Dirac function at point $\xi$  with respect to  the measure $dV_{g,\mathbb{Q}}=
\phi^{\frac{2Q}{Q-2}}dV_{0}$. Now  set
\begin{align}\label{311}
G(\xi,\eta):=\phi(\xi)^{-1}\phi(\eta)^{-1}{\tilde{G}}(\xi,\eta)+{G}_{2}
(\xi,\eta)
\end{align}
for ${\eta}\in M$, where ${G}_{2}(\xi,\eta)$ satisfies
\begin{align}\label{3.32}
L_{g,\mathbb{Q}}{G}_{2}(\xi,\cdot)
=-\phi(\xi)^{-1}\phi(\cdot)^{-\frac{Q+2}{Q-2}}{G}_{1}(\xi,\cdot).
\end{align}
$G_{2}(\xi,\cdot)$ exists since $L_{g,\mathbb{Q}}$ is invertible in $L^{2}(M)$. $G_{2}(\xi,\cdot)\in C^{\infty}(M)$ for fixed $\xi\in U$ by the
subelliptic  regularity of $L_{g,\mathbb{Q}}$. $G_{2}(\cdot,\eta)$ is also in $C^{\infty}(U)$ by differentiating (\ref{3.32}) with respect to the variable
$\xi$ repeatedly. Now we have $L_{g,\mathbb{Q}}G(\xi,\cdot)=\delta_{\xi}$, i.e.  $G(\xi,\eta)$ is the Green function  of $L_{g,\mathbb{Q}}$. By
(\ref{311}),  $G_{g,\mathbb{Q}}(\xi,\eta)-\rho_{g,\mathbb{Q}}(\xi,\eta)\in C^{\infty}(U\times U).$
\end{proof}

We have the  following transformation law of Green  functions under the conformal qc transformation.
\begin{prop}\label{p3.8}
Let $(M,g,\mathbb{Q})$ be a connected, compact, scalar positive spherical qc manifold  and $G_{g,\mathbb{Q}}$
be the Green function of the qc Yamabe operator $L_{g,\mathbb{Q}}.$ Then
\begin{align}\label{32}
G_{{\tilde{g}},\mathbb{Q}}=\frac{1}{\phi(\xi)\phi(\eta)}
G_{g,\mathbb{Q}}(\xi,\eta)
\end{align}
is the Green function of the qc Yamabe operator $L_{\tilde{g},\mathbb{Q}}$ for $\tilde{g}=\phi^{\frac{4}{Q-2}}g.$
\end{prop}

\begin{proof}
By (\ref{3.19}) and the transformation law (\ref{121}), we find that
\begin{equation*}
\begin{aligned}
\int_{M}\frac{G_{g,\mathbb{Q}}(\xi,\eta)
L_{\tilde{g},\mathbb{Q}}u(\eta)}{\phi(\xi)\phi(\eta)}
dV_{\tilde{g},\mathbb{Q}}&=\frac{1}{\phi(\xi)}\int_{M}\frac{1}{\phi(\eta)}
G_{g,\mathbb{Q}}(\xi,\eta)\phi(\eta)^{-\frac{Q+2}{Q-2}}L_{g,\mathbb{Q}}(\phi u)(\eta)
\phi(\eta)^{\frac{2Q}{Q-2}}dV_{g,\mathbb{Q}}\\
&=\frac{1}{\phi(\xi)}\int_{M}G_{g,\mathbb{Q}}(\xi,\eta)L_{g,\mathbb{Q}}(\phi u)(\eta)dV_{g,\mathbb{Q}}=u(\xi)
\end{aligned}
\end{equation*}
for any $u\in C_{0}^{\infty}(M)$. The proposition follows form the uniqueness of the Green function.
\end{proof}
\section{An invariant tensor on a scalar positive spherical qc manifold}
\subsection{A tensor invariant  under conformal qc transformations}
\begin{thm}
Let $(M,g,\mathbb{Q})$ be connected, compact, scalar positive, spherical qc manifold, which is not qc equivalent to the standard sphere. Define
\begin{align*}
can(g,\mathbb{Q}):=\mathcal{A}_{g,\mathbb{Q}}^{2}g,
\end{align*}
where
$\mathcal{A}_{g,\mathbb{Q}}^{2}(\xi)$ is defined in (\ref{1.3})
if $g=\phi^{\frac{4}{Q-2}}g_{0}$ on a neighborhood $U$ of $\xi$. Here $g_{0}$ is the standard qc metric  on $\mathscr{H}^{n}$. Then, $can(g,\mathbb{Q})$
is well-defined and depends only on the conformal class $[g]$ and $\mathbb{Q}$.
\end{thm}
\begin{proof}
We will verify that $\mathcal{A}_{g,\mathbb{Q}}$ is independent of the choice of local coordinates and $\mathcal{A}_{g,\mathbb{Q}}^{2}g$ is independent of
the choice of $g$ in the conformal class $[g]$. Suppose $\tilde{g}=\Phi^{\frac{4}{Q-2}}g$. Let $U\subset M$ be an open set and let $\rho:U\rightarrow
V\subset\mathscr{H}^{n}$ and $\tilde{\rho}:U\rightarrow \tilde{V}\subset\mathscr{H}^{n}$ be two coordinate charts such that
\begin{equation*}
\begin{aligned}
g=\rho^{*}\left(\phi_{1}^{\frac{4}{Q-2}}g_{0}\right),\
\tilde{g}=\tilde{\rho}^{*}\left(\phi_{2}^{\frac{4}{Q-2}}g_{0}\right),
\end{aligned}
\end{equation*}
for two positive function $\phi_{1}$ and $\phi_{2}$. Then, $f=\tilde{\rho}\circ\rho^{-1}: V\rightarrow \tilde{V}$ is a qc diffeomorphism of
$\mathscr{H}^{n}$ by the qc Liouville type Theorem \ref{liou} and
\begin{align*}
\left.f^{*}g_{0}\right|_{\xi'}=\phi^{\frac{4}{Q-2}}(\xi')
\left.g_{0}\right|_{\xi'}\quad {\rm with} \ \phi(\xi')=\phi_{1}(\xi')\phi_{2}^{-1}(f(\xi'))\Phi(\rho^{-1}(\xi')),
\end{align*}
for $\xi'\in V$. We claim the following the transformation law of the Green function on the quaternionic Heisenberg group under a conformal qc
transformation:
\begin{align}\label{14}
\frac{1}{\|f(\xi')^{-1}f(\eta')\|^{Q-2}}=\frac{1}{\phi(\xi')\phi(\eta')}
\cdot\frac{1}{\|\xi'^{-1}\eta'\|^{Q-2}},
\end{align}
for any $\xi', \eta'\in V.$
Apply this to $\xi'=\rho(\xi), \eta'=\rho(\eta)$ and $f=\tilde{\rho}\circ\rho^{-1}$ to get
\begin{align*}
\frac{1}{\|\tilde{\rho}(\xi)^{-1}\tilde{\rho}(\eta)\|^{Q-2}}=
\frac{1}{\phi(\rho(\xi))}\frac{1}{\phi(\rho(\eta))}
\frac{1}{\|\rho(\xi)^{-1}\rho(\eta)\|^{Q-2}}
\end{align*}
and so
\begin{equation*}
\begin{aligned}
\mathcal{A}_{\tilde{g},\mathbb{Q}}(\xi)&=\lim_{\eta\rightarrow \xi}\left|G_{\tilde{g},\mathbb{Q}}(\xi,\eta)-\frac{1}{
\phi_{2}(\tilde{\rho}(\xi))\phi_{2}(\tilde{\rho}(\eta))}\cdot\frac{C_{Q}}{\|
\tilde{\rho}(\xi)^{-1}\tilde{\rho}
(\eta)\|^{Q-2}}\right|
^{\frac{1}{Q-2}}\\
&=\lim_{\eta\rightarrow\xi}\left|\frac{G_{g,\mathbb{Q}}(\xi,\eta)}
{\Phi(\xi)\Phi(\eta)}-\frac{1}{
\Phi(\xi)\Phi(\eta)\phi_{1}(\rho(\xi))
\phi_{1}(\rho(\eta))}\cdot\frac{C_{Q}}{\|\rho(\xi)^{-1}
\rho(\eta)\|^{Q-2}}\right|
^{\frac{1}{Q-2}}\\
&=\Phi^{-\frac{2}{Q-2}}(\xi)\lim_{\eta\rightarrow\xi}
\left|G_{g,\mathbb{Q}}
(\xi,\eta)-\frac{1}{\phi_{1}({\rho}(\xi))\phi_{1}
(\rho(\eta))}\cdot\frac{C_{Q}}{\|\rho(\xi)^{-1}\rho(\eta)\|^{Q-2}}
\right|^{\frac{1}{Q-2}}\\
&=\Phi^{-\frac{2}{Q-2}}(\xi)\mathcal{A}_{g,\mathbb{Q}}(\xi).
\end{aligned}
\end{equation*}
Consequently, we have
\begin{align*}
\mathcal{A}_{\tilde{g},\mathbb{Q}}^{2}\tilde{g}
=\mathcal{A}_{g,\mathbb{Q}}^{2}{g}.
\end{align*}

It remains to check (\ref{14}). By qc Liouville type Theorem \ref{liou}, $f$ is a restriction to $U$ of a qc automorphism of $\mathscr{H}^{n}$, denoted
also by $f$. By the transformation law (\ref{121}), for functions ${\tilde{\phi}}:=\phi\circ f^{-1},\tilde{u}:=u\circ f^{-1}$ on $\tilde{V},$ we have
\begin{equation}
\begin{aligned}\label{15}
\left.L_{0}\left(\tilde{\phi}^{-1}\tilde{u}\right)\right|_{f(\eta')}&=
\left.L_{g,\mathbb{Q}}\left(\phi^{-1}u\right)\right|_{\eta'}=
\phi^{-\frac{Q+2}{Q-2}}(\eta')
\left.L_{0}(u)\right|_{\eta'},\\
\left.f^{*}dV_{0}\right|_{f(\eta')}&
=\phi^{\frac{2Q}{Q-2}}(\eta')\left.dV_{0}\right|_{\eta'}.
\end{aligned}
\end{equation}
 The first identity follows from the fact that the qc Yamabe operator is independent of the choice of coordinate charts. Then, by (\ref{15}) and taking
 transformation $f(\eta')\rightarrow\hat{\eta}$, we find that for any $u\in C_{0}^{\infty}(\mathscr{H}^{n})$
\begin{equation*}
\begin{aligned}
\int_{\mathscr{H}^{n}}\frac{C_{Q}\phi(\xi')
\phi(\eta')}{\|f(\xi')^{-1}f(\eta')\|^{Q-2}}
L_{0}u(\eta')dV_{{0}}(\eta')&=\int_{\mathscr{H}^{n}}\frac{C_{Q}\phi(\xi')}
{\|f(\xi')^{-1}f(\eta')\|^{Q-2}}
\left.L_{0}\left({\tilde{\phi}}^{-1}\tilde{u}\right)\right|_{f(\eta')}
f^{*}dV_{0}(\eta')\\
&=\int_{\mathscr{H}^{n}}
\frac{C_{Q}\phi(\xi')}{\|f(\xi')^{-1}\hat{\eta}\|^{Q-2}}
\left.L_{0}\left({\tilde{\phi}}^{-1}\tilde{u}\right)\right|_{\hat{\eta}}
dV_{0}(\hat{\eta})=u(\xi').
\end{aligned}
\end{equation*}
Now by the uniqueness of the Green  function of $L_{0}$, we find that
$$G_{0}(\xi',\eta')=\frac{C_{Q}\phi(\xi')\phi(\eta')}
{\|f(\xi')^{-1}f(\eta')\|^{Q-2}}.$$ Thus, (\ref{14}) follows. The theorem is proved.
\end{proof}
See \cite{Leutwiler} for the identity (\ref{14}) on the Euclidean space and see \cite{wang1} on the Heisenberg group.

\begin{cor}\label{c4.1}
$(M,can(g,\mathbb{Q}),\mathbb{Q})$ is a spherical qc manifold if the qc  positive mass conjecture is true.
\end{cor}
\begin{proof}
Write $g=\phi^{\frac{4}{Q-2}}g_{0}$ locally. Note that by Proposition \ref{1.10}, \ref{p3.7} and   \ref{p3.8}, we have
\begin{equation*}
\begin{aligned}
\left.G_{g,\mathbb{Q}}(\xi,\eta)^{\frac{4}{Q-2}}g\right|_{\eta}&
=\left.\left\{\rho_{g,\mathbb{Q}} (\xi,\eta)+A_{g,\mathbb{Q}}
(\xi)+O(\|\xi^{-1}\eta\|)\right\}^{\frac{4}{Q-2}}g\right|_{\eta}\\
&=\frac{\alpha(\xi)}{\|\xi^{-1}\eta\|^{4}}\left.\cdot\left\{1+\beta(\xi)
\|\xi^{-1}\eta\|^{Q-2}+O(\|\xi^{-1}\eta\|^{Q-1})\right\}^{\frac{4}{Q-2}}
g_{0}\right|_{\eta},
\end{aligned}
\end{equation*}
for  $\eta$ near $\xi,$ where $$\alpha(\xi)=\left(\frac{C_{Q}}{\phi(\xi)}\right)^{\frac{4}{Q-2}},\quad
\beta(\xi)=\frac{A_{g,\mathbb{Q}}(\xi)\phi(\xi)^{2}}{C_{Q}}.$$
Now choose the conformal qc mapping $C_{\xi}:=R\circ D_{r}\circ\tau_{\xi^{-1}}$ on a neighborhood of $\xi$  to $\mathscr{H}^{n}$ as required in the qc
positive mass conjecture, where $r=\alpha(\xi)^{-\frac{1}{2}}$. It is easy to see that $C_{\xi}(\xi)=\infty$ and
\begin{equation*}
\begin{aligned}
\left(C_{\xi}^{-1}\right)^{*}\left.\left(G_{g,\mathbb{Q}}
(\xi,\eta)^{\frac{4}{Q-2}}g\right)\right|_{\tilde{\eta}}&=
R^{*}\circ D_{r^{-1}}^{*}\circ\tau_{\xi}^{*}\left.\left(G_{g,\mathbb{Q}}
(\xi,\eta)^{\frac{4}{Q-2}}g\right)\right|_{\tilde{\eta}}\\
&=\left(1+\beta(\xi)r^{-Q+2}
\|\tilde{\eta}\|^{-Q+2}+O(\|\tilde{\eta}\|^{-Q+1})\right)^{\frac{4}{Q-2}}
g_{0}|_{\tilde{\eta}},
\end{aligned}
\end{equation*}
by Corollary \ref{333}, where $\tilde{\eta}=C_{\xi}(\eta)$.
The qc positive mass conjecture promises $\beta(\xi)r^{-Q+2}$ to be positive. So, $\mathcal{A}_{g,\mathbb{Q}}$ is non-vanishing.
\end{proof}
\begin{rem}
Let $(M,g,\mathbb{Q})$ and $(\tilde{M},\tilde{g},\tilde{\mathbb{Q}})$ be two connected, compact, scalar positive spherical qc manifolds with ${\rm
dim}M={\rm dim}\tilde{M}$, which are both not qc equivalent to the standard sphere $(S^{4n+3},g_{S},\mathbb{Q}_{S}).$ Suppose
$f:(M,g,\mathbb{Q})\rightarrow(\tilde{M},\tilde{g},\tilde{\mathbb{Q}})$ be a locally qc diffeomorphism. Then,
\begin{align*}
f^{*}can(\tilde{g},\tilde{\mathbb{Q}})(X,Y)\geq can({g},{\mathbb{Q}})(X,Y),
\end{align*}
for any tangent vector $X,Y\in TM$. This can be shown as in the CR case {\rm (cf. Proposition 3.4 in \cite{wang1})}.
\end{rem}

\subsection{Scalar positivity of the connected sum of two scalar positive spherical qc manifolds}
Scalar positive spherical qc manifolds are abundant by the following proposition.
\begin{prop}\label{51}
If $t$ is sufficiently small, the connected sum $(M_{t,\sigma,A},g,\mathbb{Q}_{t,\sigma,A})$ is scalar positive.
\end{prop}

Proposition \ref{51} follows from the following proposition.
\begin{prop}
If $t$ is sufficiently small, we have $\lambda(M_{t,\sigma,A},g,\mathbb{Q}_{t,\sigma,A})>0.$
\end{prop}

\begin{proof}
See  \cite{Kobayashi} for the Riemannian case. We provide more details that are different from the Riemannian case, compared to the proof of the CR case in \cite{wang1}.  Let
\begin{align*}
M_{0}=M_{1}\backslash\{\xi_{1}\}\cup M_{2}\backslash\{\xi_{2}\},
\end{align*}
and let $\hat{g}$ be a spherical qc metric on $M_{0}$. Then, by multiplying a positive function $\mu\in C^{\infty}(M)\backslash\{\xi_{1},\xi_{2}\}$, we
can assume $g=\mu\hat{g}$ satisfying
\begin{align*}
\left(\psi_{i}^{-1}\right)^{*}g|_{\xi}=\frac{g_{0}|_{\xi}}{\|\xi\|^{2}}\quad {\rm on} \quad B_{\mathscr{H}}(0,2)\backslash\{0\},
\end{align*}
where $\psi_{i}:U_{i}\rightarrow B_{\mathscr{H}}(0,2),\ i=1,2,$ are coordinate charts in  $(\ref{265}).$ It is easy to see that gluing mapping
$\Psi_{t,\sigma,A}$ in (\ref{2.60}) preserves the metric $\frac{g_{0}|_{\xi}}{\|\xi\|^{2}}$ on $t<\|\xi\|<\frac{1}{t}$, $0<\frac{2}{3}t<1$, by the
transformation law and Corollary \ref{333}. $\frac{g_{0}|_{\xi}}{\|\xi\|^{2}}$ is invariant under the rotation $A,$  transformation $\sigma$ and the
inversion $R.$ Hence we can glue $g$ by $\Psi_{t,\sigma,A}$ to obtain a spherical qc metric that coincides with $g$  on $M_{1}\backslash \overline{U_{1}(t)}\cup
M_{2}\backslash \overline{U_{2}(t)}.$  We denote the resulting qc metric also by $g$ by abuse of notations. We denote the connected sum by
$(M_{t},g,\mathbb{Q}_{t}).$ Here we omit the subscripts $\sigma$ and $A$ for simplicity. The scalar curvature of $g$ on
$B_{\mathscr{H}}(0,1)\backslash\{0\}$ is  $$s_{g,\mathbb{Q}}(\xi)=\frac{(Q-2)(Q+2)}{2}\cdot\frac{|y|^{2}}{\|\xi\|^{2}},\quad {\rm for}\
\xi=(y,t)\in\mathscr{H}^{n},$$ by  (\ref{8800}) in Appendix.

$(M_{0},g,\mathbb{Q})$ has two cylindrical ends. We can identify the ball with cylindrical end by the mapping
\begin{equation}
\begin{aligned}
\phi:\quad B(0,1)&\longrightarrow[0,\infty)\times\Sigma^{n}\\
\xi=D_{e^{-u}}(\eta)&\longmapsto\left(\ln\frac{1}{\|\xi\|},\frac{\xi}{\|\xi\|}\right)
=(u,\eta),
\end{aligned}
\end{equation}
where $\Sigma^{n}=\{\eta\in\mathscr{H}^{n};\|\eta\|=1\}$ is diffeomorphic to the sphere $S^{4n+2}.$ 
Define a Carnot-Carath\'eodory metric 
$$\left.\tilde{g}\right|_{\xi}=(\phi^{-1})^{*}\left(\frac{g_{0}
|_{\xi}}{\|\xi\|^{2}}\right)$$
on $[0,\infty)\times\Sigma^{n}$ and  $\tilde{\Theta}=(\tilde{\theta}_{1},\tilde{\theta}_{2},\tilde{\theta}_{3})$ is a compatible contact form.
$(B_{\mathscr{H}}(0,1)\backslash\{0\},\frac{g_{0}}{\|\xi\|^{2}},\mathbb{Q})$ is qc equivalent to $([0,\infty)\times\Sigma^{n},\tilde{g},\mathbb{Q})$.
Since $$(\psi_{i}^{-1})^{*}dV_{{g},\mathbb{Q}}=\frac{dV_{0}}{\|\xi
\|^{Q}}$$ is invariant under rescaling, it is easy to see that the measure $\tilde{\theta}_{1}\wedge\tilde{\theta}_{2}\wedge\tilde{\theta}_{3}
\wedge(d\tilde{\theta}_{s})^{2n}$ is invariant under translation $(u',\xi)\rightarrow(u'+u_{0},\xi)$ on $[0,\infty)\times\Sigma^{n}$. As a measure, we have 
\begin{align}\label{4.12}
\tilde{\theta}_{1}\wedge\tilde{\theta}_{2}
\wedge\tilde{\theta}_{3}\wedge(d\tilde{\theta}_{s})^{2n}
=dudS_{\Sigma^{n}},
\end{align}
where $dS_{\Sigma^{n}}$ is a measure on $\Sigma^{n}$. set
 $$l=\ln\frac{1}{t},$$
and write
\begin{align*}
(M_{0},g,\mathbb{Q})=([0,\infty)\times\Sigma^{n},\tilde{g},\mathbb{Q})
\cup(\hat{M},g,\mathbb{Q})
\cup([0,\infty)\times\Sigma^{n},\tilde{g},\mathbb{Q}),
\end{align*}
where $\hat{M}=M\backslash (U_{1}(1)\cup U_{2}(1))$. We identify two pieces of  $ (0,l)\times\Sigma^{n},\tilde{g},\mathbb{Q})$  to get
$(M_{t},g,\mathbb{Q}_{t})$.

Denote by $y_{\eta}$ the $y$-coordinate of $\eta\in
\Sigma^{n},$ i.e. we can write $\eta=(y_{\eta},t_{\eta})\in\mathscr{H}^{n}$ for some $t_{\eta}\in\mathbb{R}^{3}.$ Then,
\begin{align}\label{4.4}
|\nabla_{\tilde{g}}u|=\left|\nabla_{g}\left(\ln\frac{1}{\|\xi\|}\right)\right|=\frac 12 |y_{\xi}|^{2}
=\frac 12e^{-2u}|y_{\eta}|^{2},
\end{align}
where $\xi=(y_\xi,t_\xi)$ for some $t_\xi,$ and
\begin{align}\label{4.5}
s_{\tilde{g},\mathbb{Q}}(u,\eta)=\frac{(Q-2)(Q+2)}{2}{|y_{\eta}|^{2}},
\end{align}
by  (\ref{8899}) and (\ref{8800}) in the appendix.
By the definition of the Yamabe invariant $\lambda(M_{t},g,\mathbb{Q}_{t})$, we can find a positive function $f_{l}\in C^{\infty}(M_{t})$ such that
\begin{align}\label{4.14}
\int_{M_{t}}\left(b_{n}|\nabla_{g}f_{l}|^{2}
+s_{g,\mathbb{Q}_{t}}f_{l}^{2}\right)
dV_{g,\mathbb{Q}_{t}}
<\lambda(M_{t},g,\mathbb{Q}_{t})+\frac{1}{l},
\end{align}
 and
\begin{align}\label{4.15}
\int_{M_{t}}f_{l}^{\frac{2Q}{Q-2}}dV_{g,\mathbb{Q}_{t}}=1.
\end{align}
Put $A_{1}=-\min\left\{0,\min_{x\in \hat{M}}s_{\hat{g},\mathbb{Q}}\right\}{\rm Vol}(\hat{M})^{\frac{4}{Q-2}},$ which is uniformly bounded by ${\rm
Vol}(M,g)$. Thus  by using H\"older's inequality we get from (\ref{4.14}) that
$$\int_{[0,l]\times\Sigma^{n}}\left(b_{n}|
\nabla_{\tilde{g}}f_{l}|^{2}
+s_{\tilde{g},\mathbb{Q}_{t}}f_{l}^{2}\right)
dudS_{\Sigma^{n}}<\lambda(M_{t},g,\mathbb{Q}_{t})+
\frac{1}{l}+A_{1},$$
(cf. Lemma 6.2 in \cite{Kobayashi}). Note that $s_{\tilde{g},\mathbb{Q}_{t}}$ is nonnegative on $[0,\infty)\times\Sigma^{n}$ by (\ref{4.5}). Therefore,
there exists $l_{*}\in[0,l]$ such that
$$\int_{l_{*}\times\Sigma^{n}}\left(b_{n}|
\nabla_{\tilde{g}}f_{l}|^{2}
+s_{\tilde{g},\mathbb{Q}_{t}}f_{l}^{2}\right)
d{S_{\Sigma^{n}}}<\frac{\lambda(M_{t},g,\mathbb{Q}_{t})+
\frac{1}{l}+A_{1}}{l},$$
i.e. we have the estimate
\begin{align}\label{4.16}
\int_{\Sigma^{n}}\left(|
\nabla_{\tilde{g}}f_{l}(l_{*},\eta)|^{2}+|y_{\eta}|^{2}
f_{l}^{2}(l_{*},\eta)\right)dS_{\Sigma^{n}}(\eta)<\frac{C}{l},
\end{align}
by the scalar curvature of $g$ in (\ref{4.5}), where $C$ is a constant independent of $l$ (because the qc Yamabe invariants
$\lambda(M_{t},g,\mathbb{Q}_{t})$ for $t>1$ have a uniform upper bound by choosing a test function). It is different from the Riemannian case that the
scalar curvature $\frac{g_{0}|_{\xi}}{\|\xi\|^{2}}$ is not constant. But it is still independent of the variable $u.$ Now define a Lipschitz function
$F_{l}$ on $M_{0}$ by $F_{l}=f_{l}$ on $[0,l_{*})\times\Sigma^{n}\cup\hat{M}\cup[0,l-l_{*})\times\Sigma^{n}$ and
\begin{equation}
F_{l}(u,x)=\left\{
\begin{array}{rcl}
&(l_{*}+1-u)\tilde{f_{l}}(u,x)& \quad {\rm for}\quad(u,x)\in[l_{*},l_{*}+1]\times\Sigma^{n},\\
&0 &\ \ \ {\rm for}\quad(u,x)\in[l_{*}+1,\infty)\times\Sigma^{n},
\end{array}
\right.
\end{equation}
where $\tilde{f_{l}}(u,x)=f_{l}(l_{*},x)$ and similarly on $[l-l_{*} ,\infty)\times\Sigma^{n}$.

By definition, $|\nabla_{\tilde{g}}F_{l}|=|\nabla_{\tilde{g}}f_{l}|$ and  $F_{l}^{2}=f_{l}^{2}$ hold on
$[0,l_{*})\times\Sigma^{n}\cup\hat{M}\cup[0, l-l_{*} )\times\Sigma^{n}.$ On the other hand, note that
$|
\nabla_{\tilde{g}}F_{l}|\leq |
\nabla_{\tilde{g}}u||\tilde{f}_{l}|+|
\nabla_{\tilde{g}}\tilde{f}_{l}|$ pointwisely on $(l_{*},l_{*}+1)\times\Sigma^{n}$ by definition.
By (\ref{4.4}),  (\ref{4.5}) and estimate (\ref{4.16}),
we find that
\begin{equation*}
\begin{aligned}
&\int_{(l_{*},l_{*}+1)\times\Sigma^{n}}\left(b_{n}|
\nabla_{\tilde{g}}F_{l}|^{2}
+s_{\tilde{g},\mathbb{Q}}F_{l}^{2}\right)
dudS_{\Sigma^{n}}\leq C'\int_{(l_{*},l_{*}+1)\times\Sigma^{n}}
\left(|\nabla_{\tilde{g}}\tilde{f}_{l}|^{2}+|y_{\eta}|^{2}
\tilde{f_{l}}^{2}\right)d ud S_{\Sigma^{n}}(\eta)\\&\leq C''\int_{\Sigma^{n}}\left(|\nabla_{\tilde{g}}f_{l}(l_{*},\eta)|^{2}
+|y_{\eta}|^{2}f_{l}^{2}(l_{*},\eta)\right)dS_{\Sigma^{n}}(\eta)\leq\frac{B'}{l}
\end{aligned}
\end{equation*}
by $|\nabla_{\tilde{g}}\tilde{f}_{l}(u,\eta)|
=|\nabla_{\tilde{g}}\tilde{f}_{l}(l_{*},\eta)|\leq
|\nabla_{\tilde{g}}{f}_{l}(l_{*},\eta)|$ pointwisely, since horizontal subspace $H$ and $\tilde{g}$ are invariant under the translation
$(u,\eta)\rightarrow(u+u_{0},\eta)$ and $\tilde{f}_{l}$ is independent of $u.$ Therefore, we get
\begin{align*}
\int_{M_{0}}\left(b_{n}|\nabla_{g} F_{l}|^{2}
+s_{g,\mathbb{Q}}F_{l}^{2}\right)
dV_{g,\mathbb{Q}}<\lambda(M_{t},
g,\mathbb{Q}_{t})+\frac{B}{l},
\end{align*}
for some constant  $B$ independent of $l$.

Obviously from (\ref{4.15}) and the definition of $F_{l},$ we get
\begin{align*}
\int_{M_{0}}F_{l}^{\frac{2Q}{Q-2}}dV_{g,\mathbb{Q}}>1.
\end{align*}
 Therefore,
\begin{align}\label{4.20}
\inf_{F>0}\frac{\int_{M_{0}}\left(b_{n}|\nabla_{g}F|^{2}
+s_{g,\mathbb{Q}}\right)dV_{g,\mathbb{Q}}}{\left(\int_{M_{0}}
F^{\frac{2Q}{Q-2}}dV_{g,\mathbb{Q}}\right)^{\frac{Q-2}{Q}}}<
\lambda(M_{t},g,\mathbb{Q}_{t})+\frac{B}{l},
\end{align}
where the infimum is taken over all nonnegative Lipshitz functions with compact
support. It follows from the definition of the Yamabe invariant that the left side is greater than or equal to
$\lambda(M,g,\mathbb{Q})$. If $(M,g,\mathbb{Q})$ is a disjoint union of $(M_{1},g_{1},\mathbb{Q})$ and $(M_{2},g_{2},\mathbb{Q})$, we have
\begin{align}\label{4.21}
\lambda(M,g,\mathbb{Q})
=min\{\lambda(M_{1},g_{1},\mathbb{Q}),\lambda(M_{2},g_{2},\mathbb{Q})\},
\end{align}
by the definition of the qc Yamabe invariant. From (\ref{4.20}) and (\ref{4.21}), $\lambda(M_{t},g,\mathbb{Q}_{t})$ is positive if $l$ is sufficiently
large, i.e. $t$ is sufficiently small. We  complete the proof.
\end{proof}
\section{The convex cocompact discrete subgroups of ${\rm Sp}(n+1,1)$}
\subsection{Convex cocompact subgroups of ${\rm Sp}(n+1,1)$}
A group $G$ is called \emph{discrete} if the topology on $G$ is the discrete topology.
We say that $G$ acts \emph{discontinuously} on a space $X$ at point $q$ if there is a neighborhood $U$ of $q$, such that $g(U)\cap U=\emptyset$ for all
but finitely many $g\in G$.

Let $\Gamma$ be a discrete subgroup of ${\rm Sp}(n+1,1)$. Choose $q\in B^{4n+4}$.
We define the \emph{limit set} of $\Gamma$ by
\begin{align*}
\Lambda(\Gamma):=\overline{\Gamma{q}}\cap S^{4n+3},
\end{align*}
where $\overline{\Gamma{q}}$ is the closure of the orbit of $q$ under $\Gamma$.  It is known that $\Lambda(\Gamma)$ does not depend on the choice of $q\in
B^{4n+4}$ (cf. Proposition 1.4 and Proposition 2.9 in \cite{Eberlein}).
The limit set $\Lambda(\Gamma)$ of all limit points is closed and invariant under $\Gamma$.
The \emph{radial limit set of $\Gamma$} is
 $$\Lambda^{r}(\Gamma):=\left\{\xi\in\Lambda(\Gamma)\left|
\lim\inf_{T\rightarrow\infty}d(\xi_{T},\gamma(0))<\infty,\gamma\in \Gamma\right.\right\},$$
where $\xi_{T}$ refers to the point on the ray from $0$ to $\xi$ for which $d(0,\xi_{T})=T$ and $d(\cdot,\cdot)$ is the quaternionic hyperbolic distance.
\begin{align*}
\Omega(\Gamma):=S^{4n+3} \backslash\Lambda(\Gamma)
\end{align*}
is the maximal domain in $S^{4n+3}$ on which $\Gamma$ acts properly discontinuously by Proposition 8.5 in \cite{Eberlein}. $\Gamma$ is called a
\emph{Kleinian group} if $\Omega(\Gamma)$ is non empty. A Kleinian group is called \emph{elementary} if $\Lambda(\Gamma)$ contains at most two points. The
Kleinian manifold associated to $\Gamma$ is defined to be
\begin{align*}
\bar{M}_{\Gamma}=(B^{4n+4}\cup\Omega(\Gamma))/\Gamma.
\end{align*}
It is know that $\Omega(\Gamma)/\Gamma$ is a  smooth manifold and
$\bar{M}_{\Gamma}$ is  a manifold with boundary (cf. the proof of Corollary 11.11 in \cite{Eberlein}). $\Gamma$ is called \emph{convex cocompact} if
$\bar{M}_{\Gamma}$ is a compact manifold with boundary. In this case, $\Omega(\Gamma)/\Gamma$ is a compact smooth manifold.

Two basic properties of convex cocompact groups are given in the following proposition.
\begin{prop}{\rm{(cf. p. 528 in \cite{Corlette})}}
Suppose that $\Gamma $ is a convex cocompact group of ${\rm Sp}(n+1,1)$. Then,\\
(i) The radial limit set coincides with the limit set.\\
(ii) Any small deformation of the inclusion $\iota:\Gamma\rightarrow {\rm Sp}(n+1,1)$ maps $\Gamma$ isomorphically to a convex cocompact group.
\end{prop}
An interesting class of convex cocompact groups can be obtained as follows.
For a convex cocompact group $\Gamma$ in ${\rm Sp}(n+1,1)$, there is a large family of infinitesimal deformations of  $\Gamma$  and these are all
unobstructed. The small deformations give convex cocompact groups by (ii) above. 

Let $\{C_{i},C_{i}'\}_{i=1}^{k}$ be the boundary of mutually disjoint
balls $\{D_{i},D_{i}'\}_{i=1}^{k},$  where $D_{i}=B_{\mathscr{H}}(\xi_{i},r_{i})$ and $ D'_{i}=B_{\mathscr{H}}(\xi'_{i},r'_{i}),$ for some points
$\xi_{i},\xi'_{i}\in \mathscr{H}^{n},\  r_{i},\ r'_{i}>0,\ i=1,\cdots,k.$ There always exist group elements $\gamma_{i}\in {\rm Sp}(n+1,1) $ such that
$$\gamma_{i}(\mathscr{H}^{n}\backslash \bar{D_{i}})=D'_{i} $$ 
(the ``dilations" with the origin and infinity played by the centers of balls $ D_{i}$ and $D_{i}' $, respectively).  Then $\{\gamma_{i}\}$ generates a convex cocompact subgroup $\Gamma$ which is isomorphic to
the free group of rank $k$. As in the Riemannian case , we call such group the \emph{Schottky group}. It is easy to see that
$S^{4n+3}\backslash\cup_{1}^{k}(\overline{D_{i}}\cup D_{i}')$ is the fundamental domain for $\Gamma$. $\Omega(\Gamma)/\Gamma$ is   diffeomorphic to
\begin{align}
k=(S^{1}\times S^{4n+2})\sharp\cdots\sharp (S^{1}\times S^{4n+2}),
\end{align}
$\sharp$ is the connected sum operation defined by (\ref{54}).

See also \cite{Apanasov} for other interesting examples of convex cocompact groups.
\subsection{The Patterson-Sullivan measure}

\begin{thm}{\rm (cf. p. 532 in \cite{Corlette})}
For any convex cocompact Kleinian group $\Gamma\subset {\rm Sp}(n+1,1)$, there exists a probability measure $\mu_{\Gamma}$ supported on 	
$\Lambda(\Gamma)$ such that
\begin{align}\label{5.5}
\gamma^{*}\mu_{\Gamma}=|\gamma'|^{\delta(\Gamma)}\mu_{\Gamma}
\end{align}
for any $\gamma\in\Gamma,$ where $|\gamma'|$ is the conformal factor.
\end{thm}
See also \cite{EMM} and \cite{wang1} for Patterson-Sullivan measure for the complex case. We need to know  the explicit conformal  factor $|\gamma'|$ for our
purpose later.
Fix a point $q\in\mathbb{H}^{n+1}$, the series
 $\sum_{\gamma\in\Gamma}e^{-\frac{s}{2}\cdot d(p,\gamma (q))}$ 
converges for $s>\delta(\Gamma)$ for any $p\in B^{4n+4}$, and diverges for any $s<\delta(\Gamma)$. Fix a reference point $0\in B^{4n+4}$. Let us recall
the definition of Patterson-Sullivan measure in  \cite{Patterson}. Define a family of measures as
\begin{align}\label{eq:PS-measure}
\mu_{s,p}:=\frac{\sum_{\gamma\in \Gamma}e^{-\frac{1}{2}s\cdot d(p,\gamma (q))}\delta_{\gamma (q)}}{\sum_{\gamma\in \Gamma}e^{-\frac{1}{2}s\cdot d(0,\gamma
(q))}},
\end{align}
where $\delta_{\gamma (q)}$ is the Dirac measure supported at point $\gamma(q)$.
For each  $s>\delta(\Gamma)$, this  is  a  finite positive measure  concentrated  on $\Gamma q\subset\overline{\Gamma q}$. The  set  of
all  probability  measures  on $\overline{\Gamma q}$ is  compact (cf. p. 532 in \cite{Corlette}),  and so  there  is  a  sequence $s_{i}$ approaching
$\delta(\Gamma)$ from  above  such  that $\mu_{s_{i},p}$ approaches  a  limit $\mu_{s,p}$. After  rewriting  the coefficients,  we  may  assume  that  the
denominator  in  (\ref{eq:PS-measure}) diverges at $s=\delta(\Gamma)$. Thus,  we  replace  the  above  expression  by
\begin{align}\label{eq:PS-measure'}
\mu_{s,p}=\frac{\sum_{\gamma\in \Gamma}a_{\gamma}e^{-\frac{1}{2}s\cdot d(p,\gamma (q))}\delta_{\gamma (q)}}{\sum_{\gamma\in
\Gamma}a_{\gamma}e^{-\frac{1}{2}s\cdot d(0,\gamma (q))}}
\end{align}
with    $a_{\gamma}$'s so  chosen  that  the  denominator  converges  for
$s>\delta(\Gamma)$  and  diverges  for $s\leq\delta(\Gamma)$. The denominator of this expression will be denoted by $L(s,0).$
The definition of the measure $\mu_{s}$ does not depend on $p\in{B}^{4n+4}$ and the choice of $a_{\gamma}$ (cf. p. 532 \cite{Corlette}). The
\emph{Patterson-Sullivan measure} is the weak limit of these measures:
\begin{align*}
\mu_{\Gamma,p}=\lim_{s_{i}\rightarrow\delta(\Gamma)^{+}}\mu_{s_{i},p}.
\end{align*}
Since we use the right action of matrix, here $\gamma^{-1}(\gamma'(q))=(\gamma'\gamma^{-1})(q).$ For any $\gamma \in {\rm{Sp}}(n+1,1)$ and any $f\in
C(\bar{B})$, we have
\begin{equation*}
\begin{aligned}
(\gamma^{*}\mu_{s_{i},p})(f)&=\frac{\sum_{\gamma'\in\Gamma}a_{\gamma'}
e^{-\frac{1}{2}s_{i}\cdot d(p,\gamma' (q))}\gamma^{*}\delta_{\gamma' (q)}(f)}{L(s_{i},0)}\\&=\frac{\sum_{\gamma'\in
\Gamma}a_{\gamma'}e^{-\frac{1}{2}s_{i}\cdot d(\gamma^{-1} (p),\gamma^{-1}(\gamma' (q)))}f(\gamma^{-1}{(\gamma' (q))})}{L(s_{i},0)}
\\&=\frac{\sum_{\gamma'\in \Gamma}a_{\gamma'\gamma}e^{-\frac{1}{2}s_{i}\cdot d(\gamma^{-1} (p),\gamma' (q))}f({\gamma'
(q)})}{L(s_{i},0)}=\mu_{s_{i},\gamma^{-1} (p)}(f)
\end{aligned}
\end{equation*}
by the invariance of the quaternionic hyperbolic distance $d(\cdot,\cdot)$  under the action of ${\rm{Sp}}(n+1,1).$ It is easy to see that
$\{a_{\gamma'\gamma}\}$ is also such sequence satisfying  the definition for fixed $\gamma$.
Letting $s_{i}\rightarrow \delta^{+}$, we get
\begin{align*}
\gamma^{*}\mu_{\Gamma,p}=\mu_{\Gamma,\gamma^{-1} (p)}.
\end{align*}
Recall that we have the following the Radon-Nikodym relation (cf. p.77, p.81 in \cite{Yue}):
\begin{align*}
\left.\frac{d\mu_{\Gamma,\gamma^{-1} (p)}}{d\mu_{\Gamma,p}}\right|_{\xi}=
e^{-\frac{\delta}{2}\rho_{p,\xi}(\gamma^{-1}(p))},\ \xi\in S^{4n+3},
\end{align*}
where $\rho_{p,\xi}(q)$ is the Buseman function at $\xi\in\Lambda(\Gamma)$ normalized such that $\rho_{p,\xi}(p)=0.$ It follows from the formula of
Buseman function (cf. p. 81 in \cite{Yue}),
\begin{align*}
\rho_{p,\xi}(q)=\ln\frac{|1-\langle p,p\rangle||1-\langle q,\xi\rangle|^{2}}{|1-\langle q,q\rangle||1-\langle p,\xi\rangle|^{2}},
\end{align*}
that
\begin{equation*}
\begin{aligned}
\left.\frac{d\mu_{\Gamma,\gamma^{-1} (p)}}{d\mu_{\Gamma,p}}\right|_{\xi}=\left(\frac{\left|1-\langle \gamma^{-1}(p),\gamma^{-1}
(p)\rangle\right|^{\frac{1}{2}}\left|1-\langle p,\xi\rangle\right|}{\left|1-\langle p,p\rangle\right|^{\frac{1}{2}}\left|1-\langle \gamma^{-1}
(p),\xi\rangle\right|}\right)^{\delta(\Gamma)},
\end{aligned}
\end{equation*}
for $\xi\in S^{4n+3}.$ Then, for $p=0,$ we have
\begin{equation*}
\begin{aligned}
\left.\frac{d\mu_{\Gamma,\gamma^{-1} (0)}}{d\mu_{\Gamma,0}}\right|_{\xi}
&=\lim_{\eta\rightarrow\xi}\left(\frac{\left|1-\langle \gamma^{-1}(0),\gamma ^{-1}(0)\rangle\right|^{\frac{1}{2}}|1-\langle 0,\eta\rangle|}{|1-\langle
0,0\rangle|^{\frac{1}{2}}|1-\langle \gamma^{-1} (0),\eta\rangle|}\right)^{\delta(\Gamma)}\\
&=\lim_{\eta\rightarrow\xi}\left|\frac{(0,\eta)}{(\gamma^{-1} (0),\eta)}\right|^{\delta(\Gamma)}
=\lim_{\eta\rightarrow\xi}\left|\frac{(0,\eta)}{( 0,\gamma(\eta))}\right|^{\delta(\Gamma)}\\
&=\lim_{\eta\rightarrow\xi}\left|\frac{1-\langle\gamma(\eta),\gamma(\eta)\rangle}
{1-\langle\eta,\eta\rangle}\right|^{\frac{\delta(\Gamma)}{2}}
=\frac{1}{|[(\xi,1)\gamma]_{n+2}|^{\delta(\Gamma)}},
\end{aligned}
\end{equation*}
by (\ref{2.36}), where $\eta\in B^{4n+4}$ and  $|(\cdot,\cdot)|$  defined in (\ref{580})-(\ref{2.20}) is invariant under ${\rm Sp}(n+1,1).$
So if we  define $\mu_{\Gamma}:=\mu_{\Gamma,0},$ we get
$$\gamma^{*}d\mu_{\Gamma}(\xi)=\frac{1}{|[(\xi,1)\gamma]_{n+2}|
^{\delta(\Gamma)}}d\mu_{\Gamma}(\xi).$$
Then, we have
\begin{prop}\label{5.2}
For any $\gamma\in\Gamma$, the conformal factor $$|\gamma'(\xi)|=\frac{1}{|[(\xi,1)\gamma]_{n+2}|},\quad {\rm for}\ \xi\in S^{4n+3}.$$
\end{prop}

\begin{thm}\rm{(cf. p. 533 in \cite{Corlette})}\label{t5.2}
If $\Gamma\in {\rm Sp}(n+1,1)$ is a convex cocompact group which is not contained in any proper parabolic subgroup. Then the measure $\mu_{\Gamma}$
coinsides with $\delta(\Gamma)$-dimensional Hausdorff measure concentrated on $\Lambda(\Gamma),$ i.e. there exist constants $C_{1}<C_{2}$ and $R$ such
that if $\xi\in\Lambda(\Gamma)$ and $r<R$, then
\begin{align*}
C_{1}\leq\frac{\mu_{\Gamma}(B_{r}(\xi))}{r^{\delta(\Gamma)}}\leq C_{2},
\end{align*}
where $B_{r}(\xi)$ is the ball in $S^{4n+3}$ with radius $r$ under the Carnot-Carath\'eodory distance $d_{cc}$. Here
$d_{cc}(p,q)=\inf_{\gamma}\int_{0}^{1}|\gamma'(t)|dt$ for any $p,q\in S^{4n+3},$  where $\gamma:[0,1]\rightarrow S^{4n+3}$ are Lipschitzian horizontal
curves, i.e. $\gamma'(t)\in H_{\gamma(t)}$ almost everywhere.
\end{thm}

\section{An invariant qc metric of Nayatani type}

When the spherical qc manifold is $\Omega(\Gamma)/\Gamma$ for some convex cocompact subgroup $\Gamma$ of ${\rm Sp}(n+1,1)$, we can construct an invariant
qc metric $g_{\Gamma}$, which is the qc generalization of Nayatani's canonical metric in conformal geometry \cite{Nayatani}.

Define a $C^{\infty}$ function on $\Omega(\Gamma)$ by
\begin{align*}
\phi_{\Gamma}(\xi):=\left(\int_{\Lambda(\Gamma)}
G_{S}^{\kappa}(\xi,\zeta)d\mu_{\Gamma}(\zeta)\right)^{\frac{1}{\kappa}},\
\kappa=\frac{2\delta(\Gamma)}{Q-2}.
\end{align*}
Set $g_{\Gamma}:=\phi_{\Gamma}^{\frac{4}{Q-2}}g_{S}.$
Since
\begin{align*}
G_{S}(\gamma(\xi),\gamma(\zeta))=|\gamma'(\xi)|
^{-\frac{Q-2}{2}}|\gamma'(\zeta)|^{-\frac{Q-2}{2}}
G_{S}(\xi,\zeta)
\end{align*}
by the notation   $|\gamma'|$ in Proposition \ref{5.2},  Proposition \ref{1234} and the transformation law of Green functions (\ref{32}), we have
\begin{equation}\label{6.3}
\begin{aligned}
\phi_{\Gamma}(\gamma(\xi))&=
\left(\int_{\Lambda({\Gamma})}G_{S}
^{\frac{2\delta(\Gamma)}{Q-2}}(\gamma(\xi),\zeta)
d{\mu_{\Gamma}(\zeta)}\right)^{\frac{Q-2}{2\delta(\Gamma)}}
=\left(\int_{\Lambda({\Gamma})}G_{S}^
{\frac{2\delta(\Gamma)}{Q-2}}(\gamma(\xi),\gamma(\zeta))
d\gamma^{*}{\mu_{\Gamma}(\zeta)}\right)^{\frac{Q-2}{2\delta(\Gamma)}}\\
&=\left(\int_{\Lambda({\Gamma})}|\gamma'(\xi)|
^{-\delta({\Gamma})}G_{S}^{\frac{2\delta(\Gamma)}{Q-2}}(\xi,\zeta)
d{\mu_{\Gamma}(\zeta)}\right)
^{\frac{Q-2}{2\delta(\Gamma)}}
=|\gamma'(\xi)|^{-\frac{Q-2}{2}}\phi_{\Gamma}(\xi)
\end{aligned}
\end{equation}
by using  formula (\ref{5.5}). Therefore, (\ref{6.3}) together with Proposition \ref{1234} and Proposition \ref{5.2} implies that
$
\gamma^{*}g_{\Gamma}=g_{\Gamma}.
$
So it induces a spherical qc metric  on $\Omega(\Gamma)/\Gamma$.

\begin{thm}\label{6.1}
Let $\Gamma$ be a convex cocompact subgroup of ${\rm Sp}(n+1,1)$ such that $\Lambda(\Gamma)\neq\{point\}$. Then, if $\delta(\Gamma)<2n+2$ (resp.
$\delta(\Gamma)>2n+2$, resp. $\delta(\Gamma)=2n+2$), the scalar curvature of $(\Omega(\Gamma)/\Gamma,{g}_{\Gamma},\mathbb{Q})$ is positive (resp.
negative, resp. zero) everywhere.
\end{thm}

\begin{proof}
To calculate its scalar curvature,  pull back the qc  metric ${g_{\Gamma}}$ locally to the quaternionic Heisenberg group by the quaternionic Cayley
transformation.   Without loss of generality, we can assume that the southern point $(0,\cdots,0,-1)$ of the sphere $S^{4n+3}$ is contained in $\Omega(\Gamma)$. Let
$\xi\in\Omega(\Gamma)$. Under the stereographic projection $F$ defined by (\ref{2.24}), we have
\begin{align}\label{3.3}
G_{S}(F^{-1}(\xi),F^{-1}(\xi'))=
\left(\frac{1}{2\left|1+p_{n+1}\right|^{2}}\right)^{-\frac{Q-2}{4}}
\left(\frac{1}{2\left|1+p'_{n+1}\right|^{2}}\right)^{-\frac{Q-2}{4}}G_{0}(\xi,\xi'),\ \end{align}
by Corollary \ref{c2.1} and Proposition \ref{p3.7}, where $(p,p_{n+1})=F^{-1}(\xi),(p',p'_{n+1})=F^{-1}(\xi')\in S^{4n+3}.$

Define the set $\tilde{\Lambda}(\Gamma)$ and the measure $\tilde{\mu}_{\Gamma}$ on $\mathscr{H}^{n}$ by
\begin{equation}\label{3.4}
\begin{aligned}
\tilde{\Lambda}(\Gamma):=F({\Lambda}(\Gamma)),\quad
\tilde{\mu}_{\Gamma}(\xi'):=\left(\frac{1}{2\left|1+p_{n+1}'\right|^{2}}\right)^{-\frac{\delta(\Gamma)}{2}}
\left(F^{-1}\right)^{*}{\mu}_{\Gamma}(\xi'),
\end{aligned}
\end{equation}
and the metric
$\tilde{g}_{\Gamma}$ on $\mathscr{H}^{n}\backslash\tilde{\Lambda}(\Gamma)$ by
\begin{align*}
\tilde{g}_{\Gamma}=\left(F^{-1}\right)^{*}{g}_{\Gamma}.
\end{align*}
By the stereographic projection $F$, $\Gamma$ induces an action on $\mathscr{H}^{n}$. It follows that that the set $\tilde{\Lambda}(\Gamma)$ and the
metric $\tilde{g}_{\Gamma}$ on $\mathscr{H}^{n}$ are both invariant under the action of $\Gamma$.

Now we begin to calculate its scalar curvature. Write
$
\tilde{g}_{\Gamma}|_{\xi} =u^{\frac{4}{Q-2}}(\xi)g_{0}|_{\xi}.
$ with
\begin{align*}
u(\xi)&=\left(\frac{1}{2\left|1+p_{n+1}\right|^{2}}\right)^{\frac{Q-2}{4}}\left(\int_{\Lambda(\Gamma)}
G_{S}^{\kappa}(F^{-1}(\xi),\zeta)d\mu_{\Gamma}
(\zeta)\right)^{\frac{1}{\kappa}}=\left(\int_{\tilde{\Lambda}(\Gamma)}
G_{{0}}^{\kappa}(\xi,\xi')d\tilde{\mu}_{\Gamma}
(\xi')\right)^{\frac{1}{\kappa}},
\end{align*}
by using (\ref{3.3}), Corollary \ref{c2.1} and the definition of the measure $\tilde{\mu}_{\Gamma}$ in (\ref{3.4}). Denote
\begin{align*}
\tilde{g}_{\Gamma}|_{\xi}:=e^{2f(\xi)}{g}_{0}|_{\xi},\quad
\phi(\xi,\eta):=G_{0}^{\frac{2}{2-Q}}=C_{Q}^{\frac{2}{2-Q}}\|\xi^{-1} \eta\|^{2}
\end{align*}
for $\xi,\eta\in\mathscr{H}^{n}$. Then
\begin{align*}
f(\xi)=\frac{1}{\delta({\Gamma})}\ln\left(\int_{\tilde{\Lambda}(\Gamma)}
\phi(\xi,\eta)^{-\delta({\Gamma})}d\tilde{\mu}_{\Gamma}(\eta)\right),
\end{align*}
and the  scalar curvature of $\tilde{g}_{\Gamma}$ is
\begin{align*}
s_{\tilde{g}_{\Gamma},\mathbb{Q}}=2(Q+2)e^{-2f(\xi)}\left( \Delta_{0}f-\sum_{j=1}^{4n}\frac{Q-2}{4}(Y_{j}f)^{2}\right)
\end{align*}
by Corollary \ref{2.32}.
Note that,
\begin{align*}
Y_{j}f(\xi)=-\int_{\tilde{\Lambda}(\Gamma)}\phi_{\eta}^{-1}(\xi)Y_{j}
\phi_{\eta}(\xi)d\nu(\eta),
\end{align*}
where $\phi_{\eta}(\cdot):=\phi(\cdot,\eta)$ and the measure
\begin{align}
d\nu(\eta):=\left( \int_{\tilde{\Lambda}(\Gamma)}\phi_{\eta}(\xi)^{-\delta}
d\tilde{\mu}_{\Gamma}(\eta)\right)^{-1}
\phi_{\eta}^{-\delta}(\xi)d\tilde{\mu}_{\Gamma}(\eta)
\end{align}
depends on $\xi,$ where $\delta:=\delta(\Gamma)$. Then,
\begin{align*}
Y_{j}Y_{j}f(\xi)=-\int_{\tilde{\Lambda}(\Gamma)}
\phi_{\eta}^{-1}Y_{j}Y_{j}\phi_{\eta}d\nu+
(\delta+1)\int_{\tilde{\Lambda}(\Gamma)}
\left(\phi_{\eta}^{-1}Y_{j}\phi_{\eta}\right)^{2}d\nu-
\delta\left(\int_{\tilde{\Lambda}(\Gamma)}
\phi_{\eta}^{-1}Y_{j}\phi_{\eta}d\nu\right)^{2},
\end{align*}
and so
\begin{equation}
\begin{aligned}\label{514}
\Delta_{0}f-\frac{Q-2}{4}\sum_{j=1}^{4n}(Y_{j}f)^{2}
&=-\int_{\tilde{\Lambda}(\Gamma)}\phi_{\eta}^{-1}\Delta_{0}\phi_{\eta}d\nu-
\frac{(\delta+1)}{2}
\int_{\tilde{\Lambda}(\Gamma)}\phi_{\eta}^{-2}
\sum_{j=1}^{4n}|Y_{j}\phi_{\eta}|^{2}d\nu\\&\quad
+\frac
{\delta}{2}\sum_{j=1}^{4n}
\left(\int_{\tilde{\Lambda}(\Gamma)}
\phi_{\eta}^{-1}Y_{j}\phi_{\eta}d\nu\right)^{2}-
\frac{Q-2}{4}\sum_{j=1}^{4n}
\left(\int_{\tilde{\Lambda}(\Gamma)}
\phi_{\eta}^{-1}Y_{j}\phi_{\eta}d\nu\right)^{2}\\
&=\left(\frac{Q}{4}-\frac{\delta+1}{2}\right)
\left(\int_{\tilde{\Lambda}(\Gamma)}
\phi_{\eta}^{-2}\sum_{j=1}^{4n}|Y_{j}\phi_{\eta}|^{2}d\nu-
\sum_{j=1}^{4n}\left(\int_{\tilde{\Lambda}(\Gamma)}
\phi_{\eta}^{-1}Y_{j}\phi_{y}d\nu\right)^{2}\right)\\
&\quad-\frac{Q}{4}\int_{\tilde{\Lambda}
(\Gamma)}\phi_{\eta}^{-2}\sum_{j=1}^{4n}|Y_{j}\phi_{\eta}|^{2}d\nu
-\int_{\tilde{\Lambda}(\Gamma)}\phi_{\eta}^{-1}\Delta_{0}\phi_{\eta}d\nu,
\end{aligned}
\end{equation}
by $\Delta_{0}=-\frac{1}{2}\sum_{j=1}^{4n}Y_{j}Y_{j}$. Direct calculation gives
\begin{align*}
\Delta_{0}\phi_{\eta}=\frac{2}{2-Q}G_{0}^{\frac{Q}{2-Q}}\Delta_{0}
G_{0}-\frac{Q}{(2-Q)^{2}}G_{0}^{\frac{2Q-2}{2-Q}}
\sum_{j=1}^{4n}|Y_{j}G_{0}|^{2},
\end{align*}
and
\begin{align*}
\sum_{j=1}^{4n}|Y_{j}\phi_{\eta}|^{2}=\frac{4}{(2-Q)^{2}}G_{0}^{\frac{2Q}{2-Q}}
\sum_{j=1}^{4n}|Y_{j}G_{0}|^{2}.
\end{align*}
Moreover, $\Delta_{0}G_{0}(\xi,\eta)=0$ for $\xi\neq \eta$.
We see that the last two terms in (\ref{514}) are canceled. Thus,
\begin{align*}
s_{\tilde{g}_{\Gamma},\mathbb{Q}}=2(Q+2)\left(\frac{Q}{4}
-\frac{\delta+1}{2}\right)e^{-2f}
\left(\int_{\tilde{\Lambda}(\Gamma)}\phi_{\eta}^{-2}
\sum_{j=1}^{4n}|(Y_{j}\phi_{\eta}|^{2}d\nu
-\sum_{j=1}^{4n}\left(\int_{\tilde{\Lambda}(\Gamma)}\phi_{\eta}^{-1}Y_{j}
\phi_{\eta}d\nu\right)^{2}\right).
\end{align*}
Let $A=(A_{jk})$ be the symmetric matrix with
\begin{align}
A_{jk}=\int_{\tilde{\Lambda}(\Gamma)}\phi_{\eta}^{-1}Y_{j}\phi_{\eta}\cdot
\phi_{\eta}^{-1}Y_{k}\phi_{\eta}d\nu
-\int_{\tilde{\Lambda}(\Gamma)}\phi_{\eta}^{-1}Y_{j}
\phi_{\eta}d\nu\int_{\tilde{\Lambda}(\Gamma)}\phi_{\eta}^{-1}Y_{k}
\phi_{\eta}d\nu
\end{align}
Note that for real number $a_{j},\ j=1,\cdots,4n,$ we have $$\sum_{j,k=1}^{4n}A_{jk}a_{j}a_{k}=\int_{\tilde{\Lambda}(\Gamma)}h^{2}
d\nu-\left(\int_{\tilde{\Lambda}(\Gamma)}h
d\nu\right)^{2},$$
with $h=\sum_{j=1}^{4n}a_{j}\phi_{\eta}^{-1}Y_{j}
\phi_{\eta}.$ So it is easy to see that $A$ is non-negative by the Cauchy-Schwarz inequality and $\int_{\tilde{\Lambda}(\Gamma)}d\nu=1$, i.e.
$A(\xi)\geq0.$ Then,
\begin{align}\label{519}
s_{\tilde{g}_{\Gamma},\mathbb{Q}}=
8(n+2)\left(n+1-\frac{\delta}{2}\right)e^{-2f}Tr A(\xi).
\end{align}
Hence, the result follows from the fact that $Tr A(\xi)$ is nowhere vanishing on $ \Omega(\Gamma)$ by the following lemma.
\end{proof}

\begin{lem}
If $Tr A(\xi)$
vanishes at some point $\xi\in\widetilde{\Omega}(\Gamma)$ and $\delta(\Gamma)\neq 2n+2,$ then the limit set $\widetilde{\Lambda}(\Gamma)$ is exactly one
point, and $A$ vanishes identically.
\end{lem}
The proof is exactly the same as that in the CR case (cf. Lemma 4.1 in \cite{wang1}).

\appendix
\section{The Green function of the qc Yamabe operator on the quaternionic Heisenberg group}
See also \S 2.4 in \cite{Ivanov} for a similar calculation for regular solutions of the qc Yamabe equation.

\emph{Proof of Proposition \ref{1.10}.} Without loss of generality, we assume $\xi=0.$ Denote $\eta=(y,t).$ Recall that
$\Delta_{0}=-\frac{1}{2}\sum_{j=1}^{4n}Y_{j}Y_{j}.$ We have
\begin{equation}\begin{aligned}\label{app1}
Y_{4l+j}\left(\frac{1}{(\|\eta\|^{4}+\epsilon^{2})^{\frac{Q-2}{4}}}\right)&=
-\frac{Q-2}{4}\frac{Y_{4l+j}\|\eta\|^{4}}
{(\|\eta\|^{4}+\epsilon^{2})^{\frac{Q+2}{4}}},
\end{aligned}
\end{equation}
with
\begin{align}\label{app1'}
Y_{4l+j}\|\eta\|^{4}=4|y|^{2}y_{4l+j}+4\sum_{s=1}^{3}\sum_{k=1}^{4}
b_{kj}^{s}y_{4l+k}t_{s},
\end{align}
by using the expression of the vector field $Y_{4l+j}$ in (\ref{2.43}). Note that $\frac{1}{\sqrt{2}}Y_{j}$ is an orthonormal basis. Then  we get
\begin{equation}\label{a.3}
\begin{aligned}
\Delta_{0}\left(\frac{1}{(\|\eta\|^{4}+\epsilon^{2})^{\frac{Q-2}{4}}}\right) =-\frac{Q-2}{4}\left[\frac{\Delta_{0}\|\eta\|^{4}}
{(\|\eta\|^{4}+\epsilon^{2})^{\frac{Q+2}{4}}}
+\frac{(Q+2)}{4}\frac{\left|\nabla_{0}\|\eta\|^{4}\right|^{2}}
{(\|\eta\|^{4}+\epsilon^{2})^{\frac{Q+6}{4}}}\right],
\end{aligned}
\end{equation}
where
\begin{equation}\label{A2}
\begin{aligned}
2\left|\nabla_{0}\|\eta\|^{4}\right|^{2}&=\sum_{l=0}^{n-1}\sum_{j=1}^{4}
\left(Y_{4l+j}\|\eta\|^{4}\right)^{2}\\
&=16\sum_{l=0}^{n-1}\sum_{j=1}^{4}\left(|y|^{4}y_{4l+j}^{2}
+\sum_{s,s'=1}^{3}\sum_{k,k'=1}^{4}b_{k'j}^{s'}
b_{kj}^{s}y_{4l+k'}y_{4l+k}t_{s}t_{s'}\right)=16\|\eta\|^{4}|y|^{2},
\end{aligned}
\end{equation}
by using (\ref{app1'}) and $\sum_{j=1}^{4}b_{kj}^{s}b_{jk'}^{s'}=(b^{s}b^{s'})_{kk'},$  antisymmetry for $b^{s}b^{s'}$ of $s\neq s'$ and
$\left(b^{s}\right)^{2}=-{\rm id}.$
Similarly, by (\ref{app1'}), we get
\begin{equation}\label{A1}
\begin{aligned}
\Delta_{0}\|\eta\|^{4}=-\frac{1}{2} \sum_{l=0}^{n-1}\sum_{j=1}^{4}\left( {8y_{4l+j}^{2}+4|y|^{2}+8\sum_{s=1}^{3}
\sum_{k,k'=1}^{4}b_{k'j}^{s}
b_{kj}^{s}y_{4l+k'}y_{4l+k}}\right)
=-2(Q+2)|y|^{2}.
\end{aligned}
\end{equation}
Then apply (\ref{A2}) and (\ref{A1}) to (\ref{a.3}) to  get
\begin{equation}
\begin{aligned}\label{a4}
\Delta_{0}\left(\frac{1}{(\|\eta\|^{4}+\epsilon^{2})^{\frac{Q-2}{4}}}\right)
=\frac{(Q-2)(Q+2)|y|^{2}\epsilon^{2}}
{2(\|\eta\|^{4}+\epsilon^{2})^{\frac{Q+6}{4}}}.
\end{aligned}
\end{equation}
Now letting $\epsilon\rightarrow 0$,  we see that for any $u\in C^{\infty}_{0}(\mathscr{H})$,
\begin{equation*}
\begin{aligned}
\int_{R^{4n+3}} L_{0}u\cdot \frac{C_{Q}}{(|y|^{4}+|{t}|^{2})^{\frac{Q-2}{4}}}
=&\lim_{\epsilon\rightarrow 0}\int_{R^{4n+3}} L_{0}u\cdot \frac{C_{Q}}{(|y|^{4}+|{t}|^{2}+\epsilon^{2})^{\frac{Q-2}{4}}}\\
=&\lim_{\epsilon\rightarrow 0}\int_{R^{4n+3}} u\cdot C_{Q}b_{n} \Delta_{0}\left(\frac{1}{(|y|^{4}+|{t}|^{2}+\epsilon^{2}
)^{\frac{Q-2}{4}}}\right)
=u(0)
\end{aligned}
\end{equation*}
by (\ref{a4}) and the formula (\ref{146}) for $ C_{Q}^{-1}.$ The proposition  is proved. \quad \quad \quad \quad \quad \quad \quad \quad \ \ {$\Box$}

For $g|_{\xi}=\frac{g_{0}|_{\xi}}{\|\xi\|^{2}},$ we can write $g=\phi^{\frac{4}{Q-2}}g_{0}$ with $\phi={\|\xi\|^{-\frac{Q-2}{2}}}.$ It follows from the
transformation law of the scalar curvature (\ref{122}) that
\begin{equation}\begin{aligned}\label{8899}
s_{g,\mathbb{Q}}&=\phi^{-\frac{Q+2}{Q-2}}b_{n}\Delta_{0}\phi
=\|\xi\|^{\frac{Q+2}{2}}b_{n}\Delta_{0}
\|\xi\|^{-\frac{Q-2}{2}}\\
&=-\frac{Q-2}{8}b_{n}\|\xi\|^{\frac{Q+2}{2}}\left(\frac{Q+6}{8}
\|\xi\|^{-\frac{Q+14}{2}}\left|\nabla_{0}\|\xi\|^{4}\right|^{2}+\|\xi\|^{-\frac{Q+6}{2}}
\Delta_{0}\|\xi\|^{4}\right)\\&=\frac{(Q-2)(Q+2)}{2}\frac{|y|^{2}}{\|\xi\|^{2}}
\end{aligned}\end{equation}
by (\ref{A2}) and (\ref{A1}), where $\eta=(y,t)$. Similarly, we have
\begin{equation}\begin{aligned}\label{8800}
\left|\nabla_{g}\left(\ln\frac{1}{\|\xi\|}\right)\right|^{2}=\|\xi\|^{2}
\left|\nabla_{0}\left(\ln\frac{1}{\|\xi\|}\right)\right|^{2}
=\frac{1}{16\|\xi\|^{4}}\left|
\nabla_{0}\|\xi\|^{4}\right|^{2}=\frac 12 |y|^{2}.
\end{aligned}\end{equation}

\end{document}